\documentclass[12pt, a4paper]{amsart}

\usepackage[british]{babel}
\usepackage{amscd}
\usepackage{amsmath}
\usepackage{amsfonts}
\usepackage{amssymb}
\usepackage{amsthm}
\usepackage{pdfsync}
\usepackage{bbm}
\usepackage{bbold}
\usepackage{subcaption}
\usepackage{color}
\usepackage{enumerate}
\usepackage[colorlinks=true]{hyperref}
\usepackage{lmodern}
\usepackage{graphicx}
\usepackage{comment}

\definecolor{red}{rgb}{1,0,.2}

\newcommand*{\arabicdec}[1]{\the\numexpr\value{#1}\relax}

\theoremstyle{plain}
\newtheorem{theorem}{Theorem}[section]

\newtheorem{definition}[theorem]{Definition}

\newtheorem{lemma}[theorem]{Lemma}
\newtheorem{proposition}[theorem]{Proposition}

\numberwithin{equation}{section}

\newcommand{\proj}{\mathrm{proj}}
\newcommand{\iiv}{\overline{\imath}}
\newcommand{\jjv}{\overline{\jmath}}
\newcommand{\supp}{\mathrm{supp}}
\newcommand{\floor}[1]{\left\lfloor #1 \right\rfloor}

\DeclareMathOperator*{\essinf}{ess\ inf}
\allowdisplaybreaks

\begin{document}

\title{On the dimension theory of Okamoto's function}

\author{R. D\'{a}niel Prokaj}
\address{ R. D\'{a}niel Prokaj, Department of Mathematics, University of North Texas, 1155 Union Circle, Denton, TX 76203-5017, USA}
\email{rudolf.prokaj@unt.edu}

\author{Bal\'{a}zs B\'{a}r\'{a}ny}
\address{Bal\'{a}zs B\'{a}r\'{a}ny, Department of Stochastics, Institute of Mathematics, Budapest University of Technology and Economics, M\H{u}egyetem rkp. 3., Budapest, Hungary, H-1111}
\email{balubs@math.bme.hu}

\thanks{
  2020 Mathematics subject classification: 28A80, 37C45. \\
  \indent Keywords: Fractals, Self-affine sets, Hausdorff dimension, Assouad dimension.\\
  B. B\'ar\'any and R.D. Prokaj were supported by the grants NKFI K142169, and the grant NKFI KKP144059 “Fractal geometry and applications”. 
  %\indent Competing interests: The authors declare none.
}

\begin{abstract}
    In this paper, we investigate the dimension theory of the one parameter family of Okamoto's function. We compute the Hausdorff, box-counting and Assouad dimensions of the graph for a typical choice of parameter. Furthermore, we study the dimension of the level sets. We give an upper bound on the dimension of every level set, and we show that for a typical choice of parameter this value is attained for Lebesgue almost every level set.
\end{abstract}

\maketitle
\thispagestyle{empty}

\section{Introduction}\label{od32}

Okamoto \cite{okamoto2005remark} introduced and studied a one-parameter family of nowhere differentiable functions $T_a\colon[0,1]\to[0,1]$ for $a\in(0,1)$. A notable property of Okamoto's functions is that the graph is a self-affine set. That is, let $a\in(0,1)$ be arbitrary and consider the following planar iterated function system (IFS) $\mathcal{F}=\mathcal{F}_a=\{f_1,f_2,f_3\}$ on $[0,1]^2$, where
\begin{align*}
  f_1(x,y)&=\left(\frac{x}{3},ay\right),\\
  f_2(x,y)&=\left(\frac{x+1}{3},(1-2a)y+a\right),\\
  f_3(x,y)&=\left(\frac{x+2}{3},ay+1-a\right).
\end{align*}
We will often refer to $\mathcal{F}$ as the Okamoto IFS. By Hutchinson's theorem \cite{hutchinson1981fractals}, there exists a unique non-empty compact set $\mathcal{O}_a\subset[0,1]^2$ satisfying
\[
  \mathcal{O}_a=\bigcup_{i\in\{1,2,3\}}f_i(\mathcal{O}_a).
\]
The set $\mathcal{O}_a$ is the attractor of $\mathcal{F}$. It is easy to see that $\mathcal{O}_a$ defines a function as follows: for every $x\in[0,1]$, let $T_a(x)$ be the unique $y\in[0,1]$ such that $(x,y)\in\mathcal{O}_a$. One may also obtain $T_a$ and its graph $\mathcal{O}_a$ as defined by Okamoto \cite{okamoto2005remark}.

\begin{figure}
  \centering
  \begin{subfigure}{.5\textwidth}
    \centering
    \includegraphics[width=.9\linewidth]{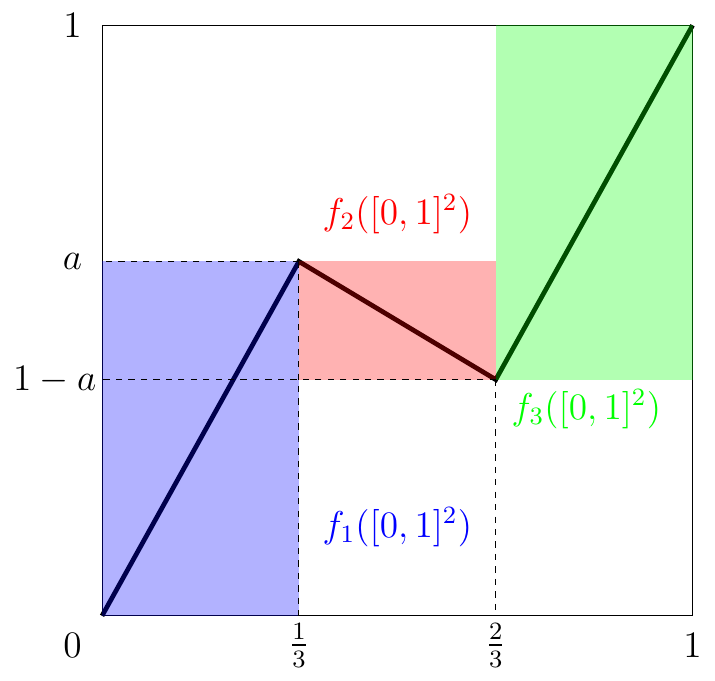}
    \caption{The Okamoto IFS}
    \label{od38}
  \end{subfigure}%
  \begin{subfigure}{.5\textwidth}
    \centering
    \includegraphics[width=.85\linewidth]{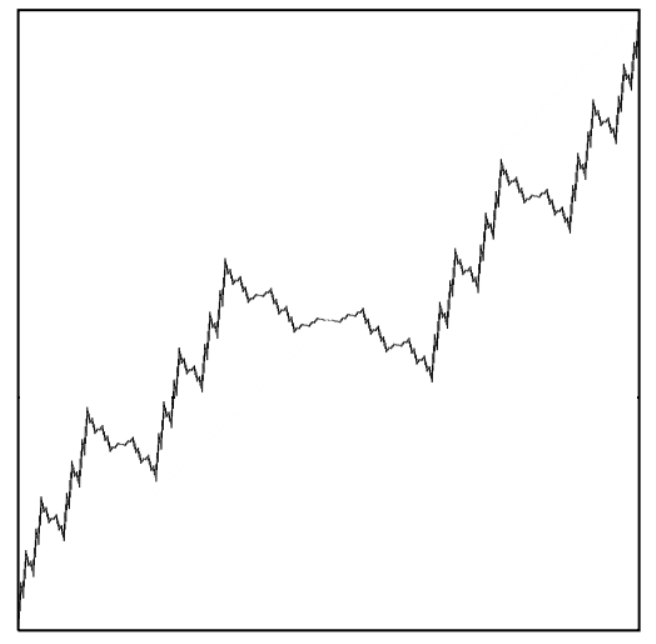}
    \caption{Okamoto's function}
    \label{od37}
  \end{subfigure}
  \caption{This figure illustrates how we obtain \\ Okamoto's function as the attractor of the IFS $\mathcal{F}$.}\label{od36}
\end{figure}

The special cases of $T_{2/3}$ and $T_{5/6}$ were studied by Perkins \cite{perkins1927elementary} and Bourbaki \cite{bourbaki2004elements} respectively, as graphs of nowhere differentiable functions.
It was Okamoto, who first studied the parameter dependence of certain properties of $T_a$.
Similarly to Perkins and Bourbaki, Okamoto also focused on the differentiability of the related functions. He showed that if $a\in(\frac{2}{3},1)$, then $T_a$ is nowhere differentiable, but if $a\in(\frac{1}{2},\frac{2}{3})$, $T_a$ is differentiable at infinitely many points.

Given the structure of Okamoto's function, one expects a strong relation between the derivative at some $x\in(0,1)$ and the ternary expansion of $x$. Assuming some technical conditions on $a\in(1/2,1)$ {and on $x\in(0,1)$}, Allaart \cite{allaart2016infinite} proved that the derivative of $T_a$ is $+\infty$ (resp. $-\infty$) at $x$ if and only if the number of $1$s in the ternary expansion of $x$ is finite and even (resp. odd). Building on his work, Dalaklis et al. \cite{dalaklis2023partial} studied the partial derivatives of Okamoto's function with respect to its defining parameter $a$ around $a=1/3$. They also found a connection between the partial derivative at $x$ and the $1$s in the ternary expansion of $x$.

Despite all the attention Okamoto's functions received over the years, not too many results are known about the fractal dimensions of $\mathcal{O}_a$. We define the dimensions that are of our main interest in Section \ref{od34}. These include the Hausdorff, box and Assouad dimensions, noted as $\dim_{\rm H}, \dim_{\rm B}$ and $\dim_{\rm A}$ respectively.

In Example~11.4 of Falconer's book \cite{whichfalconerbook}, the box dimension of the graph of general self-affine functions and in particular, the box dimension of $\mathcal{O}_a$, was calculated. The closed formula for the box-counting dimension of $\mathcal{O}_a$ was published by McCollum \cite{mccollum2010further}, who also claimed that the Hausdorff and box-counting dimension of the graph are equal. However, as Allaart pointed it out in \cite{allaart2016infinite}, his argument was incorrect.

\subsection{New results}\label{od25}

We managed to show that for typical parameters, the Hausdorff, box and Assouad dimensions of $\mathcal{O}_a$ are equal. %We denote the Hausdorff, box and Assouad dimension of a set by $\dim_{\rm H}, \dim_{\rm B}$ and $\dim_{\rm A}$ respectively, see Section~\ref{od34}.

\begin{theorem}[Main Theorem 1]\label{od70}
  Let $s_0=1+\frac{\log(4a-1)}{\log 3}$.
  There exists a set $\mathcal{E}\subset\left(\frac{1}{2},1\right)$ with $\dim_{\rm H}\mathcal{E}=0$ such that for all $a\in \left(\frac{1}{2},1\right)\setminus\mathcal{E}$ we have
  \[
    \dim_{\rm H}\mathcal{O}_a = \dim_{\rm B}\mathcal{O}_a = \dim_{\rm A} \mathcal{O}_a= s_0,
  \]
  where $\mathcal{O}_a$ is the graph of Okamoto's function defined with parameter $a$.
\end{theorem}
To prove this result, we had to verify first that the projection of the Okamoto IFS to the ${\sf y}$-axis satisfies the strong exponential separation condition. 
However, as $\mathcal{O}_a$ is the graph of a continuous function, the projected IFS is degenerate in the sense that there are strictly different symbolic codings for which the natural projection is the same for every choice of parameters. Hence, proving exponential separation for typical parameters is a non-trivial exercise.

%In Section \ref{od29}, we turn our attention to the horizontal slices of $\mathcal{O}_a$.
For $y\in(0,1)$ and $a\in(1/2,1)$, we define the corresponding level set of Okamoto's function defined with parameter $a$ as
\[
  L_y=\{x\in\mathbb{R}: (x,y)\in\mathcal{O}_a\}.
\]

In a rather recent paper, Baker and Bender \cite{baker2023cardinality} investigated the cardinality and Hausdorff dimension of level sets of Okamoto's functions. They showed that {for every $a\in(0.6077,1)\setminus \frac{\sqrt{5}-1}{2}$ and for every transcendental $a\in(0.5595,0.6077)$} if $L_y$ has continuum many points for some $y\in(0,1)$, then $\dim_{\rm H}L_y>0$. Further, if $a\in(0.5, 0.50049..)$, then one can always find a $y\in(0,1)$ such that $L_y$ only has $3$ elements, and hence the assumption on having continuum many points in a level set for positive Hausdorff dimension is clearly necessary.

Moving forward in the direction of studying the level sets, we show that for a typical parameter $a\in(1/2,1)$ every level set has Hausdorff, box-counting and Assouad dimensions at most $s_0-1$. Moreover, we also show that for a typical parameter $a\in(1/2,1)$ and Lebesgue-almost every $y\in(0,1)$, the corresponding level set's Hausdorff dimension is not just positive, but it is exactly equal to $s_0-1$.

\begin{theorem}[Main Theorem 2]\label{od68}
  Let $s_0=1+\frac{\log(4a-1)}{\log 3}$.
  There exists a set $\mathcal{E}\subset\left(\frac{1}{2},1\right)$ with $\dim_{\rm H}\mathcal{E}=0$ such that for all $a\in \left(\frac{1}{2},1\right)\setminus\mathcal{E}$ we have
  \[
    \forall y\in[0,1]: \dim_{\rm H}L_y\leq\overline{\dim}_{\rm B}L_y\leq\dim_{\rm A}L_y\leq s_0-1.
  \]
\end{theorem}

\begin{theorem}[Main Theorem 3]\label{od60}
  Let $s_0=1+\frac{\log(4a-1)}{\log 3}$.
  There exists a set $\mathcal{E}\subset\left(\frac{1}{2},1\right)$ with $\dim_{\rm H}\mathcal{E}=0$ such that for all $a\in \left(\frac{1}{2},1\right)\setminus\mathcal{E}$ we have
  \[
    \dim_{\rm H}L_y = s_0-1, \mbox{ for $\mathcal{L}^1$-almost every } y\in[0,1],
  \]
  where $\mathcal{L}^1$ denotes the one-dimensional Lebesgue measure.
\end{theorem}

We remark that $\mathcal{E}$ denotes different sets in our three main theorems. In particular, the exceptional sets in Theorem~\ref{od70} and \ref{od68} are both contained in the set defined by Theorem \ref{od97}, but we do not investigate their structure any further.

\section{Preliminaries}\label{od33}

\subsection{Elements of dimension theory}\label{od34}

Our main focus in this paper is giving a formula for the Hausdorff dimension of certain self-affine and self-similar sets and measures. We will work with iterated function systems defined either on $\mathbb{R}^2$ or on $\mathbb{R}$, and define the fractal dimensions in whole generality for $\mathbb{R}^d$ with $d\geq 1$. For further properties of these notions of fractal dimensions, we refer the reader to the books \cite{barany2023self, whichfalconerbook, fraser2020assouad}.

\begin{definition}\label{od41}
  Let $E\subset\mathbb{R}^d$ and $t\geq 0$. For $\delta>0$ we consider the Hausdorff pre-measure which is the following set function
  \begin{equation}\label{od24}
    \mathcal{H}^t_{\delta}(E) := \inf \left\{
      \sum_{i=1}^{\infty} |A_i|^t: \{A_i\}_{i=1}^{\infty}
      \mbox{ is a $\delta$-cover of $E$}
    \right\}.
  \end{equation}
  The $t$-dimensional \texttt{Hausdorff measure} of $E$ is
  \[
    \mathcal{H}^t(E) := \lim_{\delta\to 0}\mathcal{H}^t_{\delta}(E).
  \]
  We define the \texttt{Hausdorff dimension} of $E$ as
  \[
    \dim_{\rm H}E := \inf\{t:\mathcal{H}^t(E)=0\}
                  = \sup\{t:\mathcal{H}^t(E)=\infty\}.
  \]
\end{definition}

\begin{definition}\label{od14}
  Let $E\subset\mathbb{R}^d$ be a bounded set.
  For $\delta>0$, let $N_{\delta}(E)$ be the minimal number of sets of diameter $\delta$ needed to cover $E$. The \texttt{lower and upper box dimensions} of $E$ are defined by
  \begin{align*}
    \underline{\dim}_{\rm B}E &=\liminf_{\delta\to 0}\frac{\log N_{\delta}(E)}{-\log\delta}, \\
    \overline{\dim}_{\rm B}E &:=\limsup_{\delta\to 0}\frac{\log N_{\delta}(E)}{-\log\delta}.
  \end{align*}
  If the limit exists, we call it the \texttt{box dimension} of $E$ and denote it with $\dim_{\rm B}E$.
\end{definition}

Using the common notation, we write $B(x,r)$ for a closed ball of radius $r>0$ around $x\in\mathbb{R}^d$ and $d_{\mathcal{H}}$ for the Hausdorff distance.

\begin{definition}
	Let $E\subset\mathbb{R}^d$ be a bounded set. We define the \texttt{Assouad dimension} of $E$ as
	\[
	\begin{split}
	\dim_{\rm A}E=\inf\big\{\alpha>0:&\text{ there exists $C>0$ such that }\\
  &N_{r}(E\cap B(x,R))\leq C\left(\frac{R}{r}\right)^\alpha\\
	&\text{ for all }0<r<R<|E|\text{ and }x\in E\big\}.
	\end{split}
	\]
\end{definition}

\begin{definition}
    Let $E,F\subset \mathbb{R}^d$ be closed sets with $F\subset B(0,1)$. Suppose there exists a sequence of homotheties $T_k:\mathbb{R}^d\to\mathbb{R}^d$ such that $d_{\mathcal{H}}(F,T_k(E)\cap B(0,1))\to 0$ as $k\to\infty$, where $d_{\mathcal{H}}$ denotes the Hausdorff metric.
    Then $F$ is called a \texttt{weak tangent} to $E$.
\end{definition}

According to K{\"a}enm{\"a}ki, Ojala and Rossi \cite{kaenmaki2018rigidity}, we can determine the Assouad dimension of a compact set with the help of its weak tangents.
For the proof of the next proposition, see \cite[Theorem~2.3.1]{fraser2020assouad}.
\begin{proposition}\label{od08}
    Let $E\subset \mathbb{R}^d$ be a non-empty compact set. {Then the Assouad dimension} of $E$ is
    \[
      \dim_{\rm A}E = \sup\{\dim_{\rm H}F:\mbox{ $F$ is a weak tangent to $E$} \}.
    \]
\end{proposition}

The following inequalities always hold for any bounded set $E\subset\mathbb{R}^d$ \cite[Lemma~2.4.3]{fraser2020assouad}
\begin{equation}\label{cr13}
  \dim_{\rm H}E \leq\underline{\dim}_{\rm B}E\leq \overline{\dim}_{\rm B}E \leq \dim_{\rm A}E .
\end{equation}

To prove a formula for the dimension of a set $E$, we can apply results on the dimension of a measure $\mu$ supported on $E$.
\begin{definition}\label{od39}
  Let $\mu$ be a Borel probability measure on $\mathbb{R}^d$. We define the \texttt{Hausdorff dimension} of $\mu$ as
  \[
    \dim_{\rm H}\mu = \inf\{\dim_{\rm H}E: \mu(E^c)=0\},
  \]
where $E^c$ denotes the complement of the set $E$.
\end{definition}

The Hausdorff dimension of a measure $\mu$ is strongly related to the lower local dimension of $\mu$.
\begin{definition}\label{od40}
  Let $\mu$ be a Borel probability measure on $\mathbb{R}^d$ and $x\in\supp (\mu)$.
  The \texttt{lower local dimension} of the measure $\mu$ at $x$ is
  \[
    \underline{\dim}_{\rm loc}\mu(x) = \liminf_{r\to0}\frac{\log \mu(B(x,r))}{\log r},
  \]
  while its \texttt{upper local dimension} is
  \[
    \overline{\dim}_{\rm loc}\mu(x) = \limsup_{r\to0}\frac{\log \mu(B(x,r))}{\log r}.
  \]
  If the limit exists, we define the \texttt{local dimension} of $\mu$ at $x$ as
  \[
    \dim_{\rm loc}\mu(x) = \lim_{n\to\infty}\frac{\log \mu(B(x,r))}{\log r}.
  \]
\end{definition}

It is well-known that
$$
\dim_{\rm H}\mu=\essinf_{x\sim\mu}\underline{\dim}_{\rm loc}\mu(x),
$$
see for example \cite[Theorem~1.9.5]{barany2023self}.

Let us finally define the $L^q$-dimension of probability measures on $\mathbb{R}^d$ for $q>1$.

\begin{definition}\label{od02}
	Let $q\in(1,\infty)$. If $\mu$ is a probability measure on $\mathbb{R}^d$ with bounded support, then
	\[
	D(\mu,q)=\liminf_{r\to0} \frac{\log\int\mu(B(x,r))^{q-1}d\mu(x)}{(q-1)\log r}
	\]
	is the \texttt{$L^q$ dimension} of $\mu$.
\end{definition}

The next lemma is a basic application of the $L^q$-dimension. It provides estimates to the local dimension of the measure $\mu$ at every point. For a proof, see Shmerkin \cite[Lemma~1.7]{shmerkin2019furstenberg}.

\begin{lemma}[Shmerkin]\label{od01}
	Let $\mu$ be a Borel probability measure on a compact interval of $\mathbb{R}$. If $D(\mu,q)>s$ for some $q\in(1,\infty)$, then there is $r_0>0$ such that
	\[
	\mu(B(x,r))\leq r^{\left(1-\frac{1}{q}\right)s} \mbox{ for all }
	x\in\mathbb{R}, r\in(0,r_0].
	\]
	In particular, $\underline{\dim}_{\rm loc}\mu(x)\geq\left(1-1/q\right)s$ for every $x\in\mathbb{R}$.
\end{lemma}

In the upcoming sections we recall some of the most important results from the theory of self-similar and self-affine iterated function systems that are used in our proofs.

\subsection{Iterated Function Systems}

Iterated function systems (IFS) are finite lists of strict contractions $\mathcal{F}=\{f_i\}_{i=1}^m,\: m\geq 2$ defined on some metric space. Hutchinson \cite{hutchinson1981fractals} proved that there exists a unique non-empty compact set $\Lambda$ satisfying
\[
  \Lambda = \bigcup_{i=1}^m f_i(\Lambda).
\]
This set is called the \texttt{attractor} of the IFS.
We can code the points of $\Lambda$ by the elements of $\Sigma=\{1,\dots,m\}^{\mathbb{N}}$. Let $\sigma\colon\Sigma\to\Sigma$ be the usual left-shift operator. The space $(\Sigma, \sigma)$ is called the \texttt{symbolic space}, and the mapping
\[
\Pi:\Sigma\to\Lambda, \quad
\Pi(\iiv):=\lim_{n\to\infty} f_{i_1}\circ\cdots\circ f_{i_n}(0)
\]
is called the \texttt{natural projection}. Note that there is a natural metric on $\Sigma$ by $d(\iiv,\jjv)=2^{-\min\{k\geq1:i_k\neq j_k\}}$ with respect to which $\Sigma$ is a compact metric space.

We write $\Sigma_n$ for the set of $n$ length words and $\Sigma_{\ast}$ for the set of all finite words. Let us denote the length of $\iiv\in\Sigma_*$ by $|\iiv|$. For the first $n\in\mathbb{N}$ digits of $\iiv=i_1i_2\dots\in\Sigma$, we use the notation $\iiv|_n:=(i_1\dots i_n)$. For $\iiv\in\Sigma_*$, the set
\[
  [\iiv]:=\{\jjv\in\Sigma: \jjv|_{|\iiv|} = \iiv\}
\]
is called the \texttt{cylinder} of $\iiv$. For $\iiv=(i_1,\ldots,i_n)\in\Sigma_*=\bigcup_{n=0}^\infty\{1,\ldots,m\}^n$, let $f_{\iiv}$ denote the composition $f_{i_1}\circ\cdots\circ f_{i_n}$.

\subsubsection{Self-similar tools}

Since the self-similar IFSs we deal with are one dimensional, we will only recite the one dimensional version of the theorems we used in our proofs. However, most tools mentioned in this section also work with higher dimensional self-similar iterated function systems.

Fix $m\geq 2$, and let $|r_k|<1, r_k\neq 0$ and $t_k\in\mathbb{R}$ be arbitrary parameters for every $k\in\{1,\dots,m\}$.
If our iterated function system $\mathcal{F}$ is of the form
  \[
    \mathcal{F}=\left\{ f_k(\mathbf{x})=r_k\cdot \mathbf{x}+t_k\right\}_{k=1}^m,
  \]
then $\mathcal{F}$ is called \texttt{self-similar}. All the mappings in $\mathcal{F}$ are similarities of the real line. We will often write $\mathbf{r}=(r_1,\dots,r_m)$ and $\mathbf{t}=(t_1,\dots,t_m)$ for the contraction and translation vectors of $\mathcal{F}$ respectively.

\begin{definition}\label{od13}
  The \texttt{similarity dimension} of $\mathcal{F}$ is the unique number $s_0$ defined as
  \[
    \sum_{i=1}^m |r_i|^{s_0} = 1.
  \]
\end{definition}
It is a straightforward consequence of the definitions, that
\begin{equation}\label{od12}
    \dim_{\rm H}\Lambda \leq s_0,
\end{equation}
where $\Lambda$ is the attractor of the self-similar IFS $\mathcal{F}$.

\begin{definition}
  Let $\mathbf{p}=(p_1,\dots,p_m)$ be a probability vector. The \texttt{self-similar measure} of $\mathcal{F}$ with respect to $\mathbf{p}$ is a Borel probability measure $\mu$ on $\mathbb{R}^d$ such that for all Borel set $E\subset\mathbb{R^d}$
  \begin{equation}\label{od79}
    \mu(E)=\sum_{i=1}^m p_i\mu(f_i^{-1}(E)).
  \end{equation}
\end{definition}
It was proved by Hutchinson \cite{hutchinson1981fractals} that the Borel probability measure $\mu$ satisfying \eqref{od79} exists, and it is unique.
Notice that the self-similar measure $\mu$ is defined by $\mathbf{r},\mathbf{t}$ and $\mathbf{p}$.

Let us denote by $h_\mathbf{p}$ the entropy and by $\chi(\mathbf{p})$ the Lyapunov exponent. That is,
$$
h_\mathbf{p}=-\sum_{i=1}^mp_i\log p_i\text{ and }\chi(\mathbf{p})=-\sum_{i=1}^mp_i\log|r_i|.
$$
It is easy to see that
\begin{equation}\label{od12b}
	\dim_{\rm H}\mu \leq\min\left\{1,\frac{h_{\mathbf{p}}}{\chi(\mathbf{p})}\right\}.
\end{equation}

Under sufficient separation conditions, one obtains equality in \eqref{od12} and \eqref{od12b}. We define the distance of two similarity mappings
$g_1(x)=r_1x+\tau_1$ and $g_2(x)=r_2x+\tau_2$, $r_1,r_2\in (-1,1)\setminus \left\{0\right\}$, on the real line as
\begin{equation*}
  \mathrm{dist}\left(g_1,g_2\right):=
  \left\{
    \begin{array}{ll}
      |\tau_1-\tau_2|, & \mbox{if $r_1=r_2$;} \\
      \infty , & \mbox{otherwise.}
    \end{array}
  \right.
\end{equation*}

\begin{definition}
  We say that the self-similar IFS $\mathcal{F}$ satisfies the \texttt{Expo\-nential Separation Condition} (ESC) if
  there exists a $c>0 $ such that
  \begin{equation}\label{cv90}
    \mathrm{dist}(f_{\iiv}, f_{\jjv})  \geq c^{n} \mbox{ for all }
    \ell  \mbox{ and for all } \iiv,\jjv\in \left\{1, \dots ,m\right\}^{n},\ \iiv\ne\jjv
  \end{equation}
  for infinitely many $n\in\mathbb{N}$. If we one can find an $N>0$ such that \eqref{cv90} holds for every $n>N$, then we say that $\mathcal{F}$ satisfies the \texttt{Strong Exponential Separation Condition} (SESC).
\end{definition}

\begin{theorem}[Hochman {\cite[Theorem~1.1]{hochman2014self}}]\label{od07}
    Let $\mathcal{F}$ be a self-similar IFS on $\mathbb{R}$ and let $\mu$ be a self-similar measure defined by the IFS $\mathcal{F}$ and some probability vector $\mathbf{p}$. If $\mathcal{F}$ satisfies the ESC, then
    \[
      \dim_{\rm H}\mu = \min\left\{1,\frac{h_{\mathbf{p}}}{\chi(\mathbf{p})}\right\}.
    \]
    In particular, $\dim_{\rm H}\Lambda=\min\{1,s_0\}$.
\end{theorem}

\begin{theorem}[Shmerkin {\cite[Theorem~6.6]{shmerkin2019furstenberg}}]\label{od07b}
	Let $\mathcal{F}$ be a self-similar IFS on $\mathbb{R}$ and let $\mu$ be a self-similar measure defined by the IFS $\mathcal{F}$ and a probability vector $\mathbf{p}$. If $\mathcal{F}$ satisfies the ESC, then
	\[
	D(\mu,q) = \min\left\{1,\frac{\tau(q)}{q-1}\right\},
	\]
	where $\tau(q)$ is the unique solution of the equation $\sum_{i=1}^mp_i^q|r_i|^{-\tau(q)}=1$.
\end{theorem}
This theorem combined with Lemma \ref{od01} can be used to obtain a lower bound for the local dimension of self-similar measures for every point.

For a Borel probability measure $\mu$ on $\mathbb{R}$, we define its Fourier transform as
\[
  \widehat{\mu}(t)=\int_{\mathbb{R}}e^{itx}\mathrm{d}\mu(x).
\]
\begin{theorem}[Solomyak {\cite[Theorem~1.3]{solomyak2021fourier}}]\label{od10}
    There exists an exceptional set $\mathcal{E}\subset (0,1)^m$ of zero Hausdorff dimension such that for all $\mathbf{r}\in(0,1)^m\setminus\mathcal{E}$, for all choices of $\mathbf{t}$ such that the fixed points $t_i(1-r_i)^{-1}$ are not all equal and for all $(p_1,\dots,p_m)$ probability vector of non-zero entries, the corresponding self-similar measure $\mu=\mu_{\mathbf{r},\mathbf{t},\mathbf{p}}$ satisfies
    \[
      \exists \alpha>0\ \exists C>0: |\widehat{\mu}(t)|\leq C|t|^{-\alpha} \mbox{ for every $t\in\mathbb{R}$}.
    \]
\end{theorem}

\begin{theorem}[Shmerkin-Solomyak {\cite[Lemma~4.3~part~(i)]{shmerkin2016absolute}}]\label{od09}
  Let $\mu$ and $\nu$ be Borel probability measures on $\mathbb{R}$. If $\dim_{\rm H}\mu =1$ and there exist $C>0$ and $\alpha>0$ such that $|\widehat{\nu}(t)|\leq C|t|^{-\alpha}$ for every $t\in\mathbb{R}$, then
  \[
    \mu\ast\nu\ll \mathcal{L}^1,
  \]
  where $\mathcal{L}^1$ denotes the $1$-dimensional Lebesgue measure.
\end{theorem}

\subsubsection{Self-affine tools}

Let $\mathcal{F}$ be a planar \texttt{self-affine} IFS of the form
\begin{equation}\label{od06}
  \mathcal{F}=\{\mathbf{A}_i\mathbf{x}+\mathbf{t}_i\}_{i=1}^m=\{f_i(x_1,x_2)=(\alpha_ix_1, \beta_ix_2) + (t_{i,1},t_{i,2})\}_{i=1}^m,
\end{equation}
where $\mathbf{t_i}=(t_{i,1},t_{i,2})\in\mathbb{R}^2$ and $\alpha_i,\beta_i>0$ for every $i\in\{1,\dots,m\}$. We assume that $\mathcal{F}$ satisfies the \texttt{Rectangular Open Set Condition} (ROSC). In particular, $\forall i\in\{1,\dots,m\}: f_i\left([0,1]^2\right)\subset [0,1]^2$ and
\[
  i,j\in\{1,\dots,m\}, i\neq j: f_i\left((0,1)^2\right)\cap f_j\left((0,1)^2\right)=\emptyset.
\]
We write $\Lambda$ for the attractor of $\mathcal{F}$. To make sure that $\Lambda$ is not a self-similar set, we further assume that $\alpha_i\neq\beta_i$ for some $i\in\{1,\dots,m\}$. 

The most natural guess for the dimension of a self-affine set is its affinity dimension.
\begin{definition}\label{od73}
  We define the pressure function
  \begin{equation*}
    P_{\mathcal{F}}(s)=\begin{cases}
      \max\big\{ \sum_{i=1}^m |\alpha_i|^s, \sum_{i=1}^m |\beta_i|^s\big\} \mbox{, if } 0\leq s<1 \\
      \max\big\{ \sum_{i=1}^m |\alpha_i||\beta_i|^{s-1}, \sum_{i=1}^m |\beta_i||\alpha_i|^{s-1}\big\} \mbox{, if } {1\leq s<2} \\
      \sum_{i=1}^m \left(|\alpha_i||\beta_i|\right)^{s/2} \mbox{, if } 2\leq s.
    \end{cases}
  \end{equation*}
  The \texttt{affinity dimension} of $\mathcal{F}$ is the unique $s_0$ satisfying
  \[P_{\mathcal{F}}(s_0)=1,\]
  which we denote by $\dim_{\rm Aff}\mathcal{F}$.
\end{definition}
The affinity dimension is also the natural upper bound on the Hausdorff dimension of the attractor, in particular
\[
  \dim_{\rm H}\Lambda \leq\overline{\dim}_{\rm B}\Lambda\leq \dim_{\rm Aff}\mathcal{F},
\]
see \cite{FalconerMiao}.
%Observe that $\dim_{\rm Aff}\mathcal{O}_a =\max\{1,s_0\}$, where $(4a-1)(1/3)^{s_0-1}=1$.

Let $\Sigma=\{1,\dots,m\}^{\mathbb{N}}$ be the symbolic space and $\Pi:\Sigma\to [0,1]^2$ be the natural projection.
Consider an invariant and ergodic measure $\mu$ on $\left(\Sigma,\sigma\right)$. Then, its push-forward measure $\nu:=\Pi_{\ast}\mu$ is supported on $\Lambda$. Using the common notations, we write $h_\nu$ and $0<\chi_1(\nu)<\chi_2(\nu)$ for the entropy and Lyapunov exponents of $\nu$, respectively.

Since the mappings in $\mathcal{F}$ are defined by diagonal matrices, it only has two directions of contraction: one parallel to the ${\sf x}$-axis, and one parallel to the ${\sf y}$-axis. Without loss of generality we assume that the vertical contraction is weaker. That is $0<\chi_y(\nu):=\chi_1(\nu)<\chi_2(\nu)=:\chi_x(\nu)$.

We write $\proj_{\sf y}:\mathbb{R}^2\to\mathbb{R}$ for the projection onto the ${\sf y}$-axis and $\Lambda_{y_0}$ for the horizontal slice of $\Lambda$ that is projected to $y_0$
\[
  \forall y_0\in [0,1]: \Lambda_{y_0}=\proj^{-1}_y (y_0) \cap \Lambda.
\]
According to Rokhlin's theorem \cite[Theorem~9.4.11]{barany2023self}, one can disintegrate the measure $\nu$ with respect to the partition of the horizontal slices of $\Lambda$. More precisely, for $(\proj_{\sf y})_\ast \nu$-almost every $y\in[0,1]$ there exists a Borel probability measure, called conditional measure $\nu_{y}^{\proj_{\sf y}^{-1}}$, supported on $\Lambda_y$, which is uniquely defined up to a zero measure set such that for every Borel set $A\subset[0,1]^2$, $y\mapsto \nu_{y}^{\proj_{\sf y}^{-1}}(A)$ is measurable and 
\[
\nu(A)=\int \nu_{y}^{\proj_{\sf y}^{-1}}(A)d(\proj_{\sf y})_\ast \nu(y).
\]

Feng and Hu \cite{feng2009dimension} proved, that the Hausdorff dimension of $\nu$ can be computed with the help of its projection to the weak contracting direction $(\proj_{\sf y})_\ast \nu$. The theorem below is the corresponding version of \cite[Theorem~11.1.2, Theorem~11.2.1]{barany2023self}.

\begin{theorem}[Feng-Hu]\label{od15}
    Let $\nu$ be the measure defined above, and let $\nu_{y_0}^{\proj_{\sf y}^{-1}}$ denote the conditional measure of $\nu$ with respect to $\Lambda_{y_0}$. Then
    \begin{enumerate}
      \item[{\bf (1)}] for $(\proj_{\sf y})_\ast\nu$-almost every $y_0$ $$\dim_{\rm H}\nu_{y_0}^{\proj_{\sf y}^{-1}}+\dim_{\rm H}(\proj_{\sf y})_*\nu=\dim_{\rm H}\nu,$$
      \item[{\bf (2)}] $\dim_{\rm H}\nu=\frac{h_{\nu}}{\chi_2(\nu)}+\left(1-\frac{\chi_1(\nu)}{\chi_2(\nu)}\right)\dim_{\rm H}(\proj_{\sf y})_\ast\nu$.
    \end{enumerate}
\end{theorem}

The following theorem shows a nice connection between the Assouad dimension of the attractor of a self-affine IFS and its slices. We write $\proj_{\sf y}$ for the orthogonal projection to the ${\sf y}$-axis.

\begin{theorem}[Antilla-B\'{a}r\'{a}ny-K\"{a}enm\"{a}ki {\cite[Proposition~3.1]{anttila2024slices}}]\label{od18}
  Let $\mathcal{F}$ be a self-affine IFS having the form \eqref{od06} with attractor $\Lambda$. Assume that $\mathcal{F}$ satisfies the ROSC. Then
    \[
      \dim_{\rm A}\Lambda \leq \max\{
        \dim_{\rm H}\Lambda, 1+\sup_{x\in\mathbb{R}} \dim_{\rm H}\Lambda_x,
      \}
    \]
    where $\Lambda_x$ is the corresponding horizontal slice of $\Lambda$.
\end{theorem}
We note that in \cite{anttila2024slices}, this theorem is stated using a much milder separation condition the authors call a weak bounded neighborhood condition, which is always implied by the ROSC. In particular, any ball of radius $r>0$ can be intersected by uniformly finitely many disjoint open squares with side length $r$, and thus, it can be intersected by uniformly finitely many disjoint open rectangles having smallest side length $r$.

%\rk{satisfied when the cilinder sets of $\Lambda$ of a given level have disjoint interiors, and hence it is}

\section{Proving strong exponential separation}\label{od26}

Let $\mathcal{S}_a=\{S_1,S_2,S_3\}$ be the self-similar IFS that describes the projection of the graph of Okamoto's function to the ${\sf y}$-axis, where
\begin{equation}\label{od31}
  S_1(x)=ax, S_2(x)=(1-2a)x+a, S_3(x)=ax+1-a.
\end{equation}
In order to simplify the calculations, throughout this section we will work with the following system:
\begin{equation}\label{od99}
    \Phi_b=\left\{
        \phi_1(x)=\frac{1+b}{2}x-1,\: \phi_2(x)=(-b)x,\: \phi_3(x)=\frac{1+b}{2}x+1
    \right\}
\end{equation}
defined for parameters $b\in\left(0,1\right)$.

The relation between the system in \eqref{od31} and \eqref{od99} is as follows: we choose the parameter $b=2a-1$. Then for every $i\in\{1,2,3\},\: S_i=f\circ\phi_i\circ f^{-1}$ where $f(x)=\frac{1-a}{2}x+\frac{1}{2}= \frac{1-b}{4}x+\frac{1}{2}$, i.e. $\mathcal{S}_a$ is conjugated to $\Phi_b$. Thanks to this property, $S_{\iiv}=f\circ \phi_{\iiv}\circ f^{-1}$ for every $\iiv\in\bigcup_{n=0}^\infty\{1,2,3\}^{n}$. 
Hence, $\Phi_b$ is strongly exponentially separated for a given $b$ if and only if $S_a$ is strongly exponentially separated for $a=(1+b)/2$. Furthermore, due to the linear correspondence between the parameters, the Hausdorff dimension of the exceptional set of parameters $b$ for which the SESC does not hold for $\Phi_b$ equals to the Hausdorff dimension of the exceptional set of parameters $a$. Also, the attractor of $\mathcal{S}_a$ is the image of the attractor of $\Phi_b$ by the map $f$.

We write $\Lambda_b$ for the attractor of $\Phi_b$, then $\Lambda_b=\left[\frac{-2}{1-b},\frac{2}{1-b}\right]$. Moreover, let $\Pi^{\Phi}_b:\Sigma\to\Lambda_b$ denote the natural projection of the IFS $\Phi_b$
\begin{equation}\label{od98}
    \forall \iiv=i_1i_2\dots\in\Sigma : \Pi^{\Phi}_b(\iiv):=\lim_{n\to\infty} \phi_{i_1}\circ\cdots\circ\phi_{i_n}(0).
\end{equation}
With a slight abuse of notation, for a finite word $\iiv=(i_1,\ldots,i_n)\in\Sigma_*$ we will write $\Pi^{\Phi}_b(\iiv)$ for the finite composition $\phi_{i_1}\circ\cdots\circ\phi_{i_n}(0)$.

Let us also note that the region of parameters which we are interested in to study Okamoto's function is $b\in (0,1)$. However, the functions in the IFS $\Phi_b$ are strongly contractive for every $b\in (-1,1)$, and in many situations, allowing $b$ {to take on} negative values or zero is more convenient. For this reason, we will study the natural projection $\Pi^{\Phi}_b$ on the bigger parameter domain $b\in (-\varrho, \varrho)$ for an arbitrary $\varrho\in (0,1)$.

To obtain our main results, we need to show first that for most parameters $b$ the IFS $\Phi_b$ satisfies the strong exponential separation condition (SESC).
\begin{theorem}\label{od97}
    There exists a set $\mathcal{E}\subset(0,1)$ with $\dim_{\rm H}\mathcal{E}=0$ such that for all $b\in (0,1)\setminus\mathcal{E}$
    \begin{equation}\label{od89}
        \exists \varepsilon>0, \exists N\geq 1, \forall n\geq N:
        \min_{\iiv\neq\jjv\in\Sigma_n} \big\vert \Pi^{\Phi}_b(\iiv)-\Pi^{\Phi}_b(\jjv) \big\vert>\varepsilon^n.
    \end{equation}
\end{theorem}

Let us note that Theorem~\ref{od97} is not evident. Hochman \cite[Theorem~1.8]{hochman2014self} showed that for analytically parametrized non-degenerate self-similar iterated function systems the statement of Theorem \ref{od97} holds. Although the parametrization of $\Phi_b$ is clearly analytic, Hochman's result can not be applied directly, as it is a so-called degenerate IFS. That is, there are some words $\iiv,\jjv\in\Sigma, \iiv\neq\jjv$ such that $\Pi^{\Phi}_b(\iiv)=\Pi^{\Phi}_b(\jjv)$ for every parameter $b\in(0,1)$. For example,
\[
  \Pi_b^\Phi(133\dots)=\Pi_b^\Phi(211\dots), \quad \forall b\in(0,1).
\] 
We will modify Hochman's method to verify our claim.

\begin{lemma}\label{od96}
    For all $\iiv\in\Sigma$, the projection $\Pi^{\Phi}_b(\iiv)$ is an analytic function of $b$ on $(-\varrho,\varrho)$ for any $0<\varrho< 1$.
\end{lemma}

\begin{proof}
    Let $\iiv\in\Sigma$ be an arbitrary element of the symbolic space.
    Since $|b|<\varrho$ and $|\frac{1+b}{2}|<\frac{1+\varrho}{2}$, $\Pi^{\Phi}_b(\iiv|_n)$ converges uniformly for $b\in B_{\varrho}(0)\subset\mathbb{C}$. It follows that for all $\gamma\subset B_{\varrho}(0)$ closed curve in the complex plane
    \begin{equation}\label{op99}
        \lim_{n\to\infty} \int_{\gamma} \Pi^{\Phi}_b(\iiv|_n) \mathrm{d}b =
        \int_{\gamma} \Pi^{\Phi}_b(\iiv) \mathrm{d}b.
    \end{equation}
    Observe that $\Pi^{\Phi}_b(\iiv|_n)$ is just a polynomial in $b$, thus $\int_{\gamma} \Pi^{\Phi}_b(\iiv|_n) \mathrm{d}b=0$. By Morera's theorem, see for example \cite[Theorem~10.17]{Rudin}, and \eqref{op99},
    $\Pi^{\Phi}_b(\iiv)$ is an analytic function of $b$ on $B_{\varrho}(0)$.
\end{proof}

From now on we are going to work with a fixed but arbitrary $0<\varrho <1$.
For $\iiv,\jjv\in\Sigma$, let
\begin{align*}
    F_{\iiv,\jjv}^1(b) &:= b\:\Pi^{\Phi}_b(\iiv)+\frac{1+b}{2}\:\Pi^{\Phi}_b(\jjv)-1,\\
    F_{\iiv,\jjv}^2(b) &:= b\:\Pi^{\Phi}_b(\iiv)+\frac{1+b}{2}\:\Pi^{\Phi}_b(\jjv)+1,\\
    F_{\iiv,\jjv}^3(b) &:=\Delta_{\iiv,\jjv}(b):=\Pi^{\Phi}_b(\iiv)-\Pi^{\Phi}_b(\jjv).
\end{align*}
These functions take zero on certain domains due to the overlaps in the IFS. For instance, $F_{\iiv,\jjv}^1(b)\equiv 0$ when $\iiv=i_1i_2\dots, \jjv=j_1j_2\dots$ and $\forall n: i_n=1, j_n=3$. We define the following sets
\begin{align*}
    A_1 &:=\left\{ (\iiv,\jjv)\in\Sigma\times\Sigma: (i_1,j_1)\neq (1,3) \right\}\\
    A_2 &:=\left\{ (\iiv,\jjv)\in\Sigma\times\Sigma: (i_1,j_1)\neq (3,1) \right\}\\
    A_3 &:=\left\{ (\iiv,\jjv)\in\Sigma\times\Sigma: (i_1,j_1)\in\{(1,3),(3,1)\} \right\}.
\end{align*}

\begin{lemma}\label{od95}
    For any $k\in\{1,2,3\}$ and all $(\iiv,\jjv)\in A_k$, the function $b\mapsto F^k_{\iiv,\jjv}(b)$ is not the constant zero function on the interval $(-\varrho,\varrho)$ for any $0<\varrho<1$. That is,
    \[
      F^k_{\iiv,\jjv}(b)\not\equiv 0 \mbox{ on } (-\varrho,\varrho).
    \]
\end{lemma}

\begin{proof}
    First let $k=3$ and $b<0$, then choose an arbitrary $(\iiv,\jjv)\in A_3$. In this case the first cylinder intervals $\phi_1(I)$ and $\phi_3(I)$ are disjoint, thus the statement trivially holds. In particular,
    \[
      |F_{\iiv,\jjv}^3(b)|=|\Pi^{\Phi}_b(\iiv)-\Pi^{\Phi}_b(\jjv)|\geq\frac{-4b}{1-b}>-b>0.
    \]

    The remaining two cases are very similar and can be proved analogously, so we only present here the $k=1$ case.
    Let $k=1$ and let $(\iiv,\jjv)\in A_1$ be arbitrary words. According to Lemma \ref{od96}, $F_{\iiv,\jjv}^1(b)$ is analytic on $(-\varrho,\varrho)$. That is, there are coefficients $a_0,a_1,\dots$ for which
    \[
      F_{\iiv,\jjv}^1(b)=\sum_{n=0}^\infty a_nb^n.
    \]
    By the definition of $F_{\iiv,\jjv}^1(b)$, $a_0=\frac{1}{2}\Pi^{\Phi}_0(\jjv)-1$. Further, since $\Pi^{\Phi}_0(\jjv)\in[-2,2]$, we get that $a_0\leq 0$, and $a_0=0$ if and only if $\jjv=33\dots$.

    From now on we assume that $\jjv=33\dots$. We aim to show that $a_1\neq 0$ in this case, and hence $F_{\iiv,\jjv}^1(b)\not\equiv 0$ for $\iiv,\jjv\in A_1$.
    If $\jjv=33\dots$, then
    \begin{align*}
      F_{\iiv,\jjv}^1(b) &=b\Pi^{\Phi}_b(\iiv)+\frac{b+1}{2}\frac{2}{1-b}-1= \\
      &=b\Pi^{\Phi}_b(\iiv)+\frac{2b}{1-b}=b\Pi^{\Phi}_b(\iiv)+\sum_{n=1}^\infty 2b^n.
    \end{align*}
    
    Let us recall that $\sigma$ denotes the left shift on $\Sigma=\{1,2,3\}^{\mathbb{N}}$. Since $(\iiv,\jjv)\in A_1$, either $i_1=2$ or $i_1=3$.
    If $i_1=2$, then
    \[
      F_{\iiv,\jjv}^1(b)=-b^2\Pi^{\Phi}_b(\sigma\iiv)+\sum_{n=1}^\infty 2b^n,
    \]
    hence $a_1=2$, as the terms of $-b^2\Pi^{\Phi}_b(\sigma\iiv)$ are of second order or higher.
    Otherwise, $i_1=3$, and then
    \begin{align*}
      F_{\iiv,\jjv}^1(b) &=b\frac{b+1}{2}\Pi^{\Phi}_b(\sigma\iiv)+b+\sum_{n=1}^{\infty}2b^n
      \\
      &=\frac{b^2}{2}\Pi^{\Phi}_b(\sigma\iiv)+b\left(3+\frac{1}{2}\Pi^{\Phi}_b(\sigma\iiv)\right)+\sum_{n=2}^{\infty}2b^n,
    \end{align*}
    hence $a_1=3+\frac{1}{2}\Pi^{\Phi}_0(\sigma\iiv)$. As $\Pi^{\Phi}_0(\sigma\iiv)\in[-2,2]$, we have $a_1\geq 2$. It follows that the power series expansion of $F_{\iiv,\jjv}^1(b)$ always has some non-zero coefficients, and hence $F_{\iiv,\jjv}^1(b)\not\equiv 0$.
\end{proof}

\begin{lemma}\label{od94}
    There exist $c>0$ and $n>0$ such that for all $b\in(-\varrho,\varrho)$, all $k\in\{1,2,3\}$ and all $(\iiv,\jjv)\in A_k$
    \begin{equation}\label{od92}
      \exists p\in\{0,\dots,n\} \mbox{ such that }
      \bigg\vert \frac{\mathrm{d}^p}{\mathrm{d}b^p} F_{\iiv,\jjv}^k(b) \bigg\vert >c.
    \end{equation}
\end{lemma}

\begin{proof}
    We follow the lines of the proof of \cite[Proposition~5.7]{hochman2014self}. {Let $(\iiv_n)_{n\in\mathbb{N}},(\jjv_n)_{n\in\mathbb{N}}\subset\Sigma$} be sequences in the symbolic space such that $(\iiv_n,\jjv_n)\in A_k$ for all $n\in\mathbb{N}$ and $\lim_{n\to\infty} (\iiv_n,\jjv_n) = (\iiv,\jjv)$. Then for all $p\in\mathbb{N}$
    \begin{equation}\label{od93}
      \frac{\mathrm{d}^p}{\mathrm{d}b^p} F_{\iiv_n,\jjv_n}^k(b) \longrightarrow
      \frac{\mathrm{d}^p}{\mathrm{d}b^p} F_{\iiv,\jjv}^k(b)  \mbox{ uniformly on }
      (-\varrho,\varrho).
    \end{equation}

    To prove the statement of the lemma we argue by contradiction. Assume that
    for every $n\geq 1$ there exist $b_n\in(-\varrho,\varrho)$ and $(\iiv_n,\jjv_n)\in A_k$ such that
    \[
      \bigg\vert \frac{\mathrm{d}^p}{\mathrm{d}b^p} F_{\iiv_n,\jjv_n}^k(b_n) \bigg\vert <\frac{1}{n} \mbox{ for all } p\in\{0,\dots,n\}.
    \]
    Then, there exists a subsequence $\{n_l\}_{l\geq 1}$ such that for some $b\in(-\varrho,\varrho)$ and $\iiv,\jjv\in A_k$
    \[
      \lim_{l\to\infty} b_{n_l} = b,\quad
      \lim_{l\to\infty} (\iiv_{n_l},\jjv_{n_l}) = (\iiv,\jjv).
    \]
    By \eqref{od93}, $\frac{\mathrm{d}^p}{\mathrm{d}b^p} F_{\iiv,\jjv}^k(b)=0$ for all $p\in\mathbb{N}$. Since $F_{\iiv,\jjv}^k(b)$ is analytic, it implies that $F_{\iiv,\jjv}^k(b)\equiv 0$ on $(-\varrho,\varrho)$, which contradicts Lemma \ref{od95}.
\end{proof}

\begin{proposition}\label{od91}
    There exists a set $\mathcal{E}\subset (0,\varrho)$ with $\dim_{\mathrm{H}}\mathcal{E}=0$ such that for all $b\in(0,\varrho)\setminus\mathcal{E}$
    \begin{equation}\label{od90}
        \exists \delta>0, \exists N\in\mathbb{N}, \forall l\geq N,
        \forall (\iiv,\jjv)\in(\Sigma_l\times\Sigma_l)\cap A_k: \big\vert F_{\iiv,\jjv}^k(b) \big\vert>\delta^l,
    \end{equation}
    for every $k\in\{1,2,3\}$.
\end{proposition}

\begin{proof}
    Let $c>0$ and $n>0$ be the constants defined by Lemma \ref{od94}.
    Since formula \eqref{od92} holds for any $k\in \{1,2,3\}$ and $(\iiv,\jjv)\in A_k$, we may apply \cite[Lemma~5.8]{hochman2014self} to $F_{\iiv,\jjv}^k$. In particular, if $0<x<(c/2)^{2^n}$, the set $(F_{\iiv,\jjv}^k)^{-1}(-x,x)$ can be covered by $K$ many intervals of length $\tilde{c}x^{\frac{1}{2^n}}$, where $\tilde{c}\leq 2(1/c)^{\frac{1}{2^n}}$, and $K:=K(n,c)=O(1/c^n)$. That is, the set
    \[
      { \bigcup_{\iiv,\jjv\in(\Sigma_l\times\Sigma_l)\cap A_k}
      \{b: \big\vert F_{\iiv,\jjv}^k(b)\big\vert <\delta^l\}}
    \]
    can be covered by {$9^l K$} many intervals of length $\tilde{c}\delta^{\frac{l}{2^n}}$, if $l$ is large enough.

    {
    Define the set $\mathcal{E}$ as
    \begin{align*}
      \mathcal{E}=\big\{
        b\in(0,\varrho):  \forall\delta>0, \forall N\in\mathbb{N}, \exists l\geq N, \exists k\in\{1,2,3\}, \\
        \exists (\iiv,\jjv)\in(\Sigma_l\times\Sigma_l)\cap\Lambda_k
        \mbox{ such that }
        \big\vert F_{\iiv,\jjv}^k(b) \big\vert <\delta^l
      \big\}.
    \end{align*}
    Clearly, every $b\in(0,\varrho)\setminus\mathcal{E}$ satisfies \eqref{od90}.
    We can cover $\mathcal{E}$ the following way
    \[
      \mathcal{E}\subset \bigcap_{\delta>0} \mathcal{E}_{\delta} \mbox{, where }
      \mathcal{E}_{\delta}:= \bigcap_{N=1}^\infty \bigcup_{l=N}^\infty \bigcup_{\iiv,\jjv\in(\Sigma_l\times\Sigma_l)\cap A_k}
      \{b: \big\vert F_{\iiv,\jjv}^k(b)\big\vert <\delta^l\}.
    \]
    Set $M:=M(\delta):=\min\{l\geq 1: \delta^l<(c/2)^{2^n}\}$.
    We obtain
    \begin{align*}
      \mathcal{H}_{\tilde{c}\delta^{\frac{M}{2^n}}}^s(\mathcal{E}_{\delta})
      &\leq \sum_{l=M}^{\infty} \mathcal{H}_{\tilde{c}\delta^{\frac{M}{2^n}}}^s
      \left(\bigcup_{\iiv,\jjv\in(\Sigma_l\times\Sigma_l)\cap A_k}
      \{b: \big\vert F_{\iiv,\jjv}^k(b)\big\vert <\delta^l\}\right)
      \\
      &\leq \sum_{l=M}^{\infty} 9^lK\tilde{c}^s\delta^{\frac{l}{2^n}s} <\infty,
    \end{align*}
    if $9\delta^{\frac{s}{2^n}}<1$. It follows that $\dim_{\mathrm{H}}\mathcal{E}_{\delta}\leq \frac{2^n\log 9}{-\log\delta}$ for each $\delta>0$, and hence $\dim_{\rm H}\mathcal{E}=0$.
    }
\end{proof}

\begin{proof}[Proof of Theorem {\ref{od97}}]
  Pick arbitrary $0<z<\varrho<1$ and $b\in(z,\varrho)\setminus\mathcal{E}$, where $\mathcal{E}\subset (0,\varrho)$ is the set defined in Proposition \ref{od91}. Let $\delta, N$ be constants determined by \eqref{od90}. We further define
  \[
    \Gamma:=\min_{k\in\{1,2,3\}} \min_{\substack{(\iiv,\jjv)\in(\Sigma_l\times\Sigma_l)\cap A_k \\ l=0,\dots,N}}
    \left|F_{\iiv,\jjv}^k(b)\right|.
  \]
  Since $\Gamma<1$, by setting $\varepsilon:=\min\{\Gamma,\delta\}$ we have
  \begin{equation}\label{od88}
    \forall n\in\mathbb{N}, \forall (\iiv,\jjv)\in (\Sigma_n\times\Sigma_n)\cap A_k:
    \big\vert F_{\iiv,\jjv}^k(b) \big\vert >\varepsilon^n.
  \end{equation}

  To prove \eqref{od89}, we need to calculate the distance of the projections for all pairs of words, and not just the elements of $A_3$. For $n\in\mathbb{N}$ we define
  \begin{equation}\label{od87}
    \Delta_n(b):=\min_{\substack{\iiv,\jjv\in\Sigma_n \\ \iiv\neq\jjv}}
    \big\vert \Pi^{\Phi}_b(\iiv|_n)-\Pi^{\Phi}_b(\jjv|_n) \big\vert .
  \end{equation}
  Let $\iiv,\jjv\in\Sigma_n$ be the two words where the minimum in \eqref{od87} is attained.
  
  Observe that the projection of $\iiv|_n$ does not change when we write arbitrarily many $2$-s at its end. That is
  \[
    \Pi^{\Phi}_b(\iiv|_n)=\Pi^{\Phi}_b(\iiv|_n22\ldots),
  \]
   To make sure that our words do not end in $1$ and $3$ simultaneously, we introduce the sequences \[
    \iiv^{'}_{n}:=\iiv|_n22\ldots, \:\jjv^{'}_{n}:=\jjv|_n22\ldots.
  \]
  As usual, the characters of $\iiv,\jjv$ are denoted by $i_n,j_n$, while the characters of $\iiv^{'},\jjv^{'}$ are denoted by $i^{'}_n,j^{'}_n$ for $n>0$.

  Now we give lower bounds on $\Delta_n$, based only on $\iiv$ and $\jjv$. As $\iiv\neq\jjv$, we have $m:=\vert \iiv\wedge\jjv\vert<n$. If $(i_{m+1},j_{m+1})\in\{(1,3),(3,1)\}$, then
  \begin{align}\label{od86}
    \big\vert\Pi^{\Phi}_b(\iiv|_n)-\Pi^{\Phi}_b(\jjv|_n)\big\vert
    &\geq b^m\big\vert\Pi^{\Phi}_b(\sigma^m\iiv^{'}_{n+1})-\Pi^{\Phi}_b(\sigma^m\jjv^{'}_{n+1})\big\vert \\
    &\geq b^m\varepsilon^{n+1-m}
    \geq z^m\varepsilon^n
    \geq (z\varepsilon)^n. \nonumber
  \end{align}
  In the second inequality, we used that $(\sigma^m\iiv^{'}_{n},\sigma^m\jjv^{'}_{n})\in A_3$, and hence \eqref{od88} applies.

  If $(i_{m+1},j_{m+1})\in\{(1,2),(2,1)\}$, then we may assume $(i_{m+1},j_{m+1})=(2,1)$ without loss of generality. Let $q:=\min\{l\geq 0: (i^{'}_{l+m+2},j^{'}_{l+m+2})\neq (1,3)\}$. We note that {$q\leq n-m-1$} always holds, since $i^{'}_{n+1}=j^{'}_{n+1}=2$.
  {
  \begin{align*}
    \big\vert\Pi^{\Phi}_b(&\iiv|_n)-\Pi^{\Phi}_b(\jjv|_n)\big\vert
    \geq b^m\big\vert\Pi^{\Phi}_b(\sigma^m\iiv^{'}_{n})-\Pi^{\Phi}_b(\sigma^m\jjv^{'}_{n})\big\vert \\
    &= b^m\big\vert -b\Pi^{\Phi}_b(\sigma_{m+1}\iiv^{'}_{n}) -\frac{b+1}{2}\Pi^{\Phi}_b(\sigma_{m+1}\jjv^{'}_{n})+1\big\vert \nonumber \\
    &= b^m\left(\frac{b+1}{2}\right)^q\big\vert -b\Pi^{\Phi}_b(\sigma_{m+q+1}\iiv^{'}_{n}) -\frac{b+1}{2}\Pi^{\Phi}_b(\sigma_{m+q+1}\jjv^{'}_{n})+1\big\vert \nonumber \\
    &\geq b^m\left(\frac{b+1}{2}\right)^q\varepsilon^{n-(m+q+1)}
    \geq z^m\left(\frac{z}{2}\right)^q\varepsilon^n
    \geq \left(\frac{z\varepsilon}{2}\right)^n. \nonumber
  \end{align*}
    }
  The case when $(i_{m+1},j_{m+1})\in\{(2,3),(3,2)\}$ is analogous. As $0<z<\varrho<1$ were arbitrary, the proof is complete.

\end{proof}

Theorem \ref{od97} ensures that the projection of Okamoto's function to its weak contracting direction satisfies the strong exponential separation condition.  This way we can use Hochman's \cite{hochman2014self} and Shmerkin's results \cite{shmerkin2019furstenberg} to calculate the dimension of the graph of Okamoto's function.

\section{Dimensions of the graph}\label{od27}

The main goal of this chapter is calculating the Hausdorff dimension of $\mathcal{O}_a$ using the Feng-Hu Theorem~\ref{od15}, which is a generalization of the celebrated Ledrappier-Young formula to iterated function systems. It lets us reduce the problem of calculating the dimension of a set to calculating the dimension of its projections. In the case of the Okamoto's IFS, the weak contracting direction is defined by the ${\sf y}$-axis, thus we will need to work with $\proj_{\sf y} \mathcal{O}_a$.

First, we focus on the Hausdorff dimension and calculate its value for typical parameters by constructing a measure $\mu_0$ supported on $\mathcal{O}_a$ for which\[
  \dim_{\rm H}\mu_0 = \dim_{\rm Aff}\mathcal{O}_a.
\]

We code the points of $\mathcal{O}_a$ with the elements of the symbolic space $\Sigma=\{1,2,3\}^{\mathbb{N}}$. For $n\in\mathbb{N}$, let us write $\Sigma_n=\{1,2,3\}^n$ for the set of length $n$ words and $\Sigma_{\ast}=\bigcup_{n=0}^{\infty}\{1,2,3\}^n$ for the set of all finite words. The function $\Pi^{\mathcal{F}}_a:\Sigma\to[0,1]^2$ that relates the words of the symbolic space to the attractor $\mathcal{O}_a$ is called the \texttt{natural projection}
\begin{equation}\label{od69}
  \forall \iiv=i_1i_2\dots\in\Sigma : \Pi^{\mathcal{F}}_a(\iiv):=\lim_{n\to\infty} f_{i_1}\circ\cdots\circ f_{i_n}(0).
\end{equation}
For two words $\iiv,\jjv\in\Sigma$, we denote their common initial part with $\iiv\wedge\jjv$ and its length with $|\iiv\wedge\jjv|$. As usual, we endow the symbolic space with the metric
\[
  d(\iiv,\jjv)=2^{-\sup\{n:|\iiv\wedge\jjv|=n\}}.
\]

\subsection{Hausdorff dimension of the graph}\label{od78}

Recall that our original planar IFS $\mathcal{F}_a$ consists of the functions
\begin{align*}
  f_1(x,y)&=\left(\frac{x}{3},ay\right),\\
  f_2(x,y)&=\left(\frac{x+1}{3},(1-2a)y+a\right),\\
  f_3(x,y)&=\left(\frac{x+2}{3},ay+1-a\right),
\end{align*}
where $a\in(1/2,1)$ is the parameter of the system. We write $\mathcal{O}_a$ for the attractor of this IFS and $\Pi^{\mathcal{F}}_a:\Sigma\to[0,1]^2$ for the natural projection with respect to this IFS.

Let $\mu_0:=\mu_0(a)$ be the natural measure of the Okamoto IFS. In particular, $\mu_0$ can be obtained by taking the push-forward by $\Pi^{\mathcal{F}}_a$ of a Bernoulli measure $\nu$ on $\Sigma$ defined with probabilities
\[
  p_1=a\left(\frac{1}{3}\right)^{s_0-1},
  p_2=(2a-1)\left(\frac{1}{3}\right)^{s_0-1},
  p_3=a\left(\frac{1}{3}\right)^{s_0-1},
\]
where $s_0$ is the unique number satisfying
\begin{equation}\label{od80}
  (4a-1)\left(\frac{1}{3}\right)^{s_0-1}=1.
\end{equation}
Observe that by Definition \ref{od73} $\dim_{\rm Aff}\mathcal{O}_a=s_0$, thus
\begin{equation}\label{od72}
  \dim_{\rm H}\mathcal{O}_a\leq\overline{\dim}_{\rm B}\mathcal{O}_a\leq \dim_{\rm Aff}\mathcal{O}_a=s_0.
\end{equation}

Let $\proj_{\sf y}:\mathbb{R}^2\to\mathbb{R}$ be the orthogonal projection to the ${\sf y}$-axis
\begin{equation}\label{od63}
  \forall (x,y)\in \mathbb{R}^2: \proj_{\sf y}(x,y)=y.
\end{equation}
In Section \ref{od26}, we introduced the self-similar IFS $\mathcal{S}_a=\{ax, (1-2a)x+a,ax+1-a\}, a\in(1/2,1)$ as the projection of $\mathcal{F}$ to the ${\sf y}$-axis. The measure $(\proj_{\sf y}\circ\Pi^{\mathcal{F}}_a)_{\ast}\nu$ is a self-similar measure of $\mathcal{S}_a$ possessing the following nice property.
\begin{lemma}\label{od23}
  There exists an exceptional set of parameters $\mathcal{E}\subset\left(\frac{1}{2},1\right)$ with $\dim_{\rm H}\mathcal{E}=0$ such that for every $a\in\left(\frac12,1\right)\setminus\mathcal{E}$ and every $q>1$
   \[D((\proj_{\sf y}\circ\Pi^{\mathcal{F}}_a)_{\ast}\nu,q)=1.\]
  In particular,
  \[
    \underline{\dim}_{\rm loc}(\proj_{\sf y}\circ\Pi^{\mathcal{F}}_a)_{\ast}\nu(y)\geq1,
  \]
  {for every $y\in(0,1)$.}
\end{lemma}

Before proving this lemma, we show how it implies $\dim_{\rm H}\mathcal{O}_a=s_0$.
Fix a parameter $a\in \left(\frac{1}{2},1\right)\setminus\mathcal{E}$ for now. As a consequence of Lemma~\ref{od23}, the Hausdorff dimension of $(\proj_{\sf y}\circ\Pi^{\mathcal{F}}_a)_{\ast}\nu$ is also greater than or equal to $1$, therefore
\begin{equation}\label{od21}
	\dim_{\rm H}(\proj_{\sf y}\circ\Pi^{\mathcal{F}}_a)_{\ast}\nu=1.
\end{equation}

In fact, the measure $(\proj_{\sf y}\circ\Pi^{\mathcal{F}}_a)_{\ast}\nu$ is the orthogonal projection of the natural measure $\mu_0$ perpendicular to the weak contracting direction of the IFS $\mathcal{F}$, and thus by Theorem \ref{od15}
\begin{align}\label{od77}
  \dim_{\rm H}\mu_0 &= \frac{h_{\mu_0}}{\chi_2}+\left(1-\frac{\chi_1}{\chi_2}\right)\dim_{\rm H}(\proj_{\sf y})_{\ast}\mu_0 \\
  &= \frac{h_{\mu_0}}{\chi_2}+\left(1-\frac{\chi_1}{\chi_2}\right)\dim_{\rm H}(\proj_{\sf y}\circ\Pi^{\mathcal{F}}_a)_{\ast}\nu = 1+\frac{h_{\mu_0} -\chi_1}{\chi_2},
\end{align}
where $h_{\mu_0}$ is the entropy and $0<\chi_1<\chi_2$ are the Lyapunov exponents of the measure $\mu_0$. It is easy to calculate these values, as $\mu_0=(\Pi^{\mathcal{F}}_a)_{\ast}\nu$ is a push-forward of a Bernoulli measure, and $\Pi_a^{\mathcal{F}}$ is a bijection outside a set of zero measure.
\begin{gather}\label{od76}
  h_{\mu_0} =-\sum_{i=1}^3p_i\log p_i
  = -\left(\frac{1}{3}\right)^{s_0-1} \bigg(
    (4a-1)\log\left(\frac{1}{3}\right)^{s_0-1} \\
    +\: 2a\log(a)+(2a-1)\log(2a-1) \bigg) \nonumber
\end{gather}
\begin{align}\label{od75}
  \chi_1 &=-\left(
  2a\left(\frac{1}{3}\right)^{s_0-1}\log a + (2a-1)\left(\frac{1}{3}\right)^{s_0-1}\log (2a-1)
  \right) \\
  &= -\left(\frac{1}{3}\right)^{s_0-1} \bigg(
    2a \log a + (2a-1)\log (2a-1)
  \bigg) \nonumber
\end{align}
\begin{equation}\label{od74}
  \chi_2=-\sum_{i=1}^3 p_i\log \frac{1}{3} = \log 3
\end{equation}
By substituting \eqref{od76},\eqref{od75} and \eqref{od74} back to \eqref{od77}, we obtain that
\begin{equation}\label{od71}
  \dim_{\rm H}\mu_0 = s_0,
\end{equation}
if $a\in \left(\frac{1}{2},1\right)\setminus\mathcal{E}$.
Since $\mu$ is supported on $\mathcal{O}_a$,
\begin{equation}\label{od20}
  \dim_{\rm H}\mathcal{O}_a =\dim_{\rm B}\mathcal{O}_a =s_0,
\end{equation}
by \eqref{od72} and \eqref{od71}.

Now we provide a proof of Lemma \ref{od23}.
\begin{proof}[Proof of Lemma {\ref{od23}}]
  %In order to make our calculations more compact, we will work with the conjugate system $\Phi_b$ again, just like in Section~\ref{od26}. In particular, $b=2a-1$ and the probabilities that define $\nu$ are
  %\begin{equation}\label{od22}
  %  p_1=\frac{b+1}{2}\left(\frac{1}{3}\right)^{s-1},
  %  p_2=b\left(\frac{1}{3}\right)^{s-1},
  %  p_3=\frac{b+1}{2}\left(\frac{1}{3}\right)^{s-1},
  %\end{equation}
  %where $(2b+1)\left(\frac{1}{3}\right)^{s-1}=1$.
  Due to the discussion in the beginning of Section~\ref{od26} and by Theorem~\ref{od97}, there exists a set $\mathcal{E}\subset(1/2,1)$ such that for every $a\in(1/2,1)\setminus\mathcal{E}$ the IFS $\mathcal{S}_a$ in \eqref{od31} satisfies the SESC and $\dim_{\mathrm{H}}\mathcal{E}=0$. By Theorem~\ref{od07b}, we can calculate the $L^q$ dimension of $(\proj_{\sf y}\circ\Pi^{\mathcal{F}}_a)_\ast\nu$ for any $q\in(1,\infty)$ using the formula
  \begin{equation}\label{od84}
    \forall a\in\left(\frac{1}{2},1\right)\setminus\mathcal{E}:
    D((\proj_{\sf y}\circ\Pi^{\mathcal{F}}_a)_\ast\nu,q)=\min\left\{\frac{\tau(q)}{q-1},1\right\},
  \end{equation}
  where $\tau(q)$ is defined as the unique number satisfying
  \begin{equation}\label{od83}
    2\left(a\left(\frac{1}{3}\right)^{s_0-1}\right)^q\left(a\right)^{-\tau(q)}+
    \left((2a-1)\left(\frac{1}{3}\right)^{s_0-1}\right)^q (2a-1)^{-\tau(q)} = 1.
  \end{equation}
  By rearranging the terms of \eqref{od83}, we arrive to
  \begin{equation}\label{od82}
    \left(\frac{1}{3}\right)^{(s-1)q}\left(2a^{q-\tau(q)}+(2a-1)^{q-\tau(q)}\right)=1.
  \end{equation}

  Since $s_0$ is defined as the unique number satisfying $(4a-1)\left(\frac{1}{3}\right)^{s_0-1}=1$, by substituting $q-1$ to $\tau(q)$ in \eqref{od82}, we get
  \begin{align}\label{od81}
    \left(\frac{1}{3}\right)^{(s_0-1)q}
    &\big(2a^{q-(q-1)}
    +(2a-1)^{q-(q-1)}\big) = \\
    &=(4a-1)\left(\frac{1}{3}\right)^{(s_0-1)q}=\left(\frac{1}{3}\right)^{(s_0-1)(q-1)}<1, \nonumber
  \end{align}
  for every $q>1$. Moreover, observe that the mapping
  \[
    \tau\mapsto\left(\frac{1}{3}\right)^{(s_0-1)q}\left(2a^{q-\tau}+(2a-1)^{q-\tau}\right)
  \]
  is increasing and continuous, and it tends to infinity as $\tau\to\infty$. This observation and \eqref{od81} together imply that $\tau(q)>q-1$ for all $q>1$. Thus according to \eqref{od84},
  \[
    \forall q>1, \forall a\in(\frac{1}{2},1) \setminus\mathcal{E}:\quad
    D((\proj_{\sf y}\circ\Pi^{\mathcal{F}}_a)_\ast\nu, q)=1.
  \]
  Applying Lemma \ref{od01} gives
  \[
    \underline{\dim}_{\rm loc}(\proj_{\sf y}\circ\Pi^{\mathcal{F}}_a)_{\ast}\nu(x)\geq 1\text{ for all }x.
  \]  
\end{proof}

\subsection{Upper bound on the Assouad dimension}

We are left to show that $\dim_{\rm A}\mathcal{O}_a\leq s_0$. To do this, we first need to bound the Hausdorff dimension of the level sets from above. Recall, for $y\in(0,1)$ and $a\in(1/2,1)$, we defined the corresponding level set of Okamoto's function defined with parameter $a$ as
\[
  L_y=\{x\in\mathbb{R}: (x,y)\in\mathcal{O}_a\}.
\]

\begin{lemma}\label{od19}
  Let $\mu_0$ be the natural measure of the Okamoto IFS. Then
  \begin{equation}\label{od64}
    \underline{\dim}_{\rm loc}((\proj_{\sf y})_{\ast}\mu_0,y)\geq 1  \implies
    \dim_{\rm H}L_y\leq s_0-1.
  \end{equation}
\end{lemma}

\begin{proof}
  Set $\lambda_1=\lambda_3=a$ and $\lambda_2=2a-1$. With the help of this notation we can define a symbolic cover of the attractor $\mathcal{O}_a$. For $r\in(0,1)$ let
  \begin{equation}\label{od67}
    \mathcal{M}_{r}:=\{\iiv=(i_1,\dots,i_n)\in\Sigma:
    \lambda_{i_1}\cdots\lambda_{i_n}\leq r<
    \lambda_{i_1}\cdots\lambda_{i_{n-1}}\}
  \end{equation}
  We will use the shorthand notation $\lambda_{\iiv}=\lambda_{i_1}\cdot\dots\cdot\lambda_{i_n}$ for $\iiv=(i_1\dots i_n)\in\Sigma_*$. The length of the words contained in $\mathcal{M}_{r}$ must be bigger than some constant multiplier of the logarithm of $r$, precisely
  \begin{equation}\label{od66}
    \forall \iiv\in\mathcal{M}_r:
    |\iiv|\geq\frac{\log r}{\log (2a-1)}.
  \end{equation}

  Let $\varepsilon>0$ and $y\in[0,1]$ be arbitrary. We may assume without loss of generality that for this $y$ we have $\underline{\dim}_{\rm loc}((\proj_{\sf y})_{\ast}\mu_0,y)\geq 1$, which implies
  \begin{equation}\label{od62}
    \exists R>0, \forall r<R:
    (\proj_{\sf y})_{\ast}\mu_0(B_r(y))\leq r^{1+\frac{\varepsilon\log 3}{2\log (2a-1)}},
  \end{equation}
  where $B_r(y)$ denotes the closed ball of radius $r$ around $y$.
  On the other hand, using \eqref{od66} we obtain
  \begin{align*}\label{od61}
    (\proj_{\sf y})_{\ast}\mu_0(B_r(y)) &\geq \sum_{\substack{\iiv\in \mathcal{M}_r\\ \iiv:y\in\phi_{\iiv}[0,1]}} \lambda_{\iiv}\left(\frac{1}{3}\right)^{(s_0-1)|\iiv|}\geq
    (2a-1)r \sum_{\substack{\iiv\in \mathcal{M}_r\\ \iiv:y\in\phi_{\iiv}[0,1]}} \left(\frac{1}{3}\right)^{(s_0-1)|\iiv|} \\
    &\geq (2a-1)r \sum_{\substack{\iiv\in \mathcal{M}_r\\ \iiv:y\in\phi_{\iiv}[0,1]}} \left(\frac{1}{3}\right)^{(s_0-1+\varepsilon)|\iiv|} \left(\frac{1}{3}\right)^{-\varepsilon|\iiv|} \nonumber\\
    &\geq (2a-1)r \sum_{\substack{\iiv\in \mathcal{M}_r\\ \iiv:y\in\phi_{\iiv}[0,1]}} \left(\frac{1}{3}\right)^{(s_0-1+\varepsilon)|\iiv|} r^{\frac{\varepsilon\log 3}{\log (2a-1)}}.\nonumber \\
    &\geq (2a-1)r\cdot r^{\frac{\varepsilon\log 3}{\log (2a-1)}}
    \mathcal{H}^{s_0-1+\varepsilon}_{r^{\frac{-\log 3}{\log (2a-1)}}}(L_y)
    \nonumber
  \end{align*}
  In the last step, we used \eqref{od24} and that $\{\iiv\in\mathcal{M}_r:\ y\in\phi_{\iiv}[0,1]\}$ defines a covering of $L_y$ over the ${\sf x}$-axis, and $\left(\frac{1}{3}\right)^{|\iiv|}\leq r^{\frac{-\log 3}{\log (2a-1)}}$. It is immediate that
  \begin{equation}
    \forall y\in(0,1)\ \forall \varepsilon>0: \dim_{\rm H}L_y\leq {s_0-1+\varepsilon .}
  \end{equation}
  The choice of $\varepsilon$ was arbitrary, thus the statement of the theorem follows.
\end{proof}

Lemma~\ref{od23}, Lemma \ref{od19} and Theorem \ref{od18} together yields that for every $a\in\left(\frac12,1\right)\setminus\mathcal{E}$
\begin{equation}\label{od17}
  \dim_{\rm A}\mathcal{O}_a \leq s_0.
\end{equation}
Theorem \ref{od70} follows as a consequence of \eqref{od20} and \eqref{od17}.

\subsection{Upper bound for Level sets}\label{od29}

It follows from Lemma~\ref{od23} and Lemma \ref{od19} that for typical parameter $a$ the Hausdorff dimension of any horizontal slice of $\mathcal{O}_a$ is less than or equal to $s_0-1$. Now we show that the same can be said about their Assouad dimension.

\begin{lemma}\label{od05}
  For every $y\in[0,1]$ we have
  \[
    \dim_{\rm A}L_y \leq\sup_{y\in[0,1]}\dim_{\rm H}L_y.
  \]
\end{lemma}

\begin{proof}
  Pick an arbitrary level set $L_y$, and let $E\neq\emptyset$ be a weak tangent to it. There exists a sequence of similarities $(T_k)_{k\geq 1}$ such that
  \begin{equation}\label{od85}
    E_k:=T_k(L_y)\cap B(0,1)\to E \mbox{ as } k\to\infty
  \end{equation}
  with respect to the Hausdorff metric.
  Since $T_k$ is a similarity for every $k\geq 1$, it has the form
  \[
    T_k(x,y)=\frac{x-x_k}{r_k},
  \]
  for suitable constants $x_k$ and $r_k$.
  {
  It follows from \eqref{od85} that there exists a subsequence $(k_l)_{l\geq 1}$ for which both $(x_{k_l})_{l\geq 1}$ and $(r_{k_l})_{l\geq 1}$ converge.
  }
  By Proposition \ref{od08}, it is enough to show that
  \begin{equation}\label{od04}
    \dim_{\rm H}E \leq\sup_{y\in[0,1]}\dim_{\rm H}L_y,
  \end{equation}
  as $E$ is an arbitrary weak tangent. To show this we will embed $E$ into some level sets of $\mathcal{O}_a$.

  First assume that there exists $0<a,b$ constants for which
  \[
    \forall k\geq 1:\quad 0<a\leq r_k\leq b<\infty.
  \]
  {
  Then, $(T_{k_l}^{-1})_{l\geq 1}$ converges to a similarity mapping $T$ in supremum norm.
  }
  Further, $T(E)\subset L_y$. It follows that
  \[
    \dim_{\rm H}E =\dim_{\rm H}T(E) \leq\dim_{\rm H} L_y.
  \]

  {
  If $\lim_{l\to\infty}|r_{k_l}|=\infty$, $E$ must be a singleton, and \eqref{od04} trivially holds. We are left to deal with the case of $\lim_{l\to\infty}|r_{k_l}|=0$. Since every $T_{k_l}$ is expansive, we can always find a suitable $n_l\geq 1$ such that $E_{k_l}$ is the image of the level sets of some cylinders of level $n_l$ under $T_{k_l}$. Namely,
  \begin{equation}\label{od03}
    \exists L>0, \forall l\geq L, \exists n_l \mbox{ such that }
    \left(\frac{1}{3}\right)^{n_l} \leq 2r_{k_l} < \left(\frac{1}{3}\right)^{n_l-1}.
  \end{equation}
  Thanks to the structure of Okamoto's function, there are at most $2$ cylinder sets of level $n_l$ satisfying this condition. We write $\mathbf{i}_l$ and $\mathbf{j}_l$ for the two words of length $n_l$ in the symbolic space that code these cylinders. In particular, for $\mathbf{i}_l, \mathbf{j}_l\in \Sigma_{n_l}$
  \begin{align*}
    f_{\mathbf{i}_l}\left([0,1]^2\right)\cap \big\{T_{k_l}^{-1}(E_{k_l})\times\{y\}\big\}\neq\emptyset, \\
    f_{\mathbf{j}_l}\left([0,1]^2\right)\cap \big\{T_{k_l}^{-1}(E_{k_l})\times\{y\}\big\}\neq\emptyset.
  \end{align*}

  Define $F_l:=L_y\cap B(x_{k_l},r_{k_l})$ and consider the sets $f^{-1}_{\mathbf{i}_l}(F_l\times\{y\})$ and $f^{-1}_{\mathbf{j}_l}(F_l\times\{y\})$. Since \eqref{od03} holds, we can find a subsequence $(l_m)_{m\geq 1}$ that satisfies the following two properties
  \begin{enumerate}
    \item There exist sets $C_{y^{'}}, C_{y^{''}}$ such that
    \[
      f^{-1}_{\mathbf{i}_{l_m}}(F_{l_m}\times\{y\})\to C_{y^{'}} \mbox{, and }
      f^{-1}_{\mathbf{j}_{l_m}}(F_{l_m}\times\{y\})\to C_{y^{''}}.
    \]
    \item There exist similarities $g, h$ such that
    \[
      f^{-1}_{\mathbf{i}_{l_m}}\circ T^{-1}_{k_{l_m}}\to g \mbox{, and }
      f^{-1}_{\mathbf{j}_{l_m}}\circ T^{-1}_{k_{l_m}}\to h.
    \]
  \end{enumerate}
  By compactness, $C_{y^{'}}\subset L_{y^{'}}$ and $C_{y^{''}}\subset L_{y^{''}}$ for suitable level sets. Further, $g(E)\subset C_{y^{'}}$ and $h(E)\subset C_{y^{''}}$.
  }
  It follows that
  \[
    \dim_{\rm H}E =\dim_{\rm H}(g(E)\cup h(E))
    \leq \dim_{\rm H}(L_{y^{'}}\cup L_{y^{''}})\leq\sup_{y\in[0,1]}\dim_{\rm H}L_y.
  \]

\end{proof}

We obtain Theorem~\ref{od68} by combining Lemma~\ref{od19} and Lemma~\ref{od05}.

\section{Dimension of Lebesgue typical slices}\label{od28}

This section is dedicated to the proof of our third main theorem, Theorem~\ref{od60}.
Remember that $\mathcal{S}_a=\{S_1,S_2,S_3\}$ is the self-similar IFS that describes the projection of the graph of Okamoto's function to the ${\sf y}$-axis.
\begin{equation*}\label{od59}
  S_1(x)=ax, S_2(x)=(1-2a)x+a, S_3(x)=ax+1-a
\end{equation*}
By Theorem \ref{od97}, $\mathcal{S}_a$ is strongly exponentially separated for all parameters $a\in(1/2,1)\setminus\mathcal{E}$, where $\mathcal{E}$ is a small set of exceptional parameters with $\dim_{\rm H}\mathcal{E}=0$. Set $p:=(2a-1)\left(\frac{1}{3}\right)^{s_0-1}$. With the help of the probability $p$, we define a homogeneous subsystem of higher iterates of $\mathcal{S}_a$.
For an $m\in\mathbb{N}$, we define
\begin{equation}\label{od58}
  \mathcal{M}_m:=\{\iiv\in\Sigma_m: \#_2\iiv=\floor{mp}\},\: \mathcal{S}_m=\mathcal{S}_{m,a}:=\{S_{\iiv}\}_{\iiv\in\mathcal{M}_m},
\end{equation}
where $\#_2\iiv$ denotes the number of $2$ digits in $\iiv$.
There are {$|\mathcal{M}_m|=2^{m-\floor{pm}}\binom{m}{\floor{pm}}$} many functions in $\mathcal{S}_m$, and they all share the same contraction ratio {$\lambda=\lambda(a):=a^{m-\floor{pm}}(1-2a)^{\floor{pm}}$}. We write $\mathcal{K}_m$ for the attractor of $\mathcal{S}_m$.

As $\mathcal{S}_m$ is a subsystem of $\mathcal{S}^m=\{S_{\iiv}\}_{\iiv\in\Sigma_m}$, it also satisfies the strong exponential separation condition for all parameters outside a set of zero Hausdorff dimension.

We are going to approximate $\mathcal{O}_a$ with the help of subsystems defined by the alphabets $\mathcal{M}_m, m\in\mathbb{N}$. Let $\Pi_{m,a}:\Sigma\to[0,1]$ be the natural projection with respect to $\mathcal{S}_m$, and let $\nu_{m,a}$ be the uniform measure on $\mathcal{M}_m^{\mathbb{N}}$.
The push-forward measure $\mu_{m,a}:=\left(\Pi_{m,a}\right)_{\ast}\nu_{m,a}$ is supported on $\mathcal{K}_m$.

The next proposition claims that $\mu_{m,a}$ is absolutely continuous with respect to the one-dimensional Lebesgue measure, hence $\dim_{\rm H}\mu_{m,a}=1$ for most parameters $a\in(1/2,1)$.
\begin{proposition}\label{od57}
  {For $m\in\mathbb{N}$ sufficiently large}, there exists a set $\mathcal{E}\subset\left(\frac{1}{2},1\right)$ with $\dim_{\rm H}\mathcal{E}=0$ for which
  \[
    \forall a\in \left(\frac{1}{2},1\right)\setminus\mathcal{E}:
    \mu_{m,a}\ll \mathcal{L}^1,
  \]
  where $\mathcal{L}^1$ denotes the one-dimensional Lebesgue measure.
\end{proposition}

\begin{proof}
  Fix an arbitrary but large $k>1$. We define the following two self-similar iterated function systems
  \begin{align*}
    \mathcal{S}_{(<k),m} &= \left\{g_{\pmb{\jjv}}(x):=\lambda^kx+
    \sum_{l=1}^{k-1}\lambda^{l-1}S_{\jjv_l}(0)\right\}_{\pmb{\jjv}=(\jjv_1,\dots,\jjv_{k-1})\in\mathcal{M}_m^{k-1}} \\
    \mathcal{S}_{(=k),m} &= \bigg\{h_{\jjv}(x)=\lambda^kx+\lambda^{k-1}S_{\jjv}(0)\bigg\}_{\jjv\in\mathcal{M}_{m}}.
  \end{align*}
  Notice that the natural pressure $\Pi_{m,a}$ takes the following form 
  \begin{equation}\label{eq:plus}
  \begin{split}
    \Pi_{m,a}(\pmb{\jjv}) &= \sum_{i=1}^{\infty} \lambda^{i-1}S_{\jjv_i}(0) =
    \sum_{i=0}^{\infty}\lambda^{ik} \left( \sum_{l=1}^{k-1} \lambda^{l-1}S_{\jjv_{ik+l}}(0) + \lambda^{k-1}S_{\jjv_{(i+1)k}}(0)\right)  \\
    &= \sum_{i=0}^{\infty}\lambda^{ik} \left( \sum_{l=1}^{k-1} \lambda^{l-1}S_{\jjv_{ik+l}}(0)\right) 
    +\sum_{i=0}^{\infty}\lambda^{ik}\left(\lambda^{k-1}S_{\jjv_{(i+1)k}}(0)\right) \\
    &= \Pi^{(<k)}_m(\pmb{\jjv}) + \Pi^{(=k)}_m(\pmb{\jjv})
  \end{split}
\end{equation}
  for every $\pmb{\jjv}\in \mathcal{M}_m^{\mathbb{N}}$, where $\Pi^{(<k)}_m$ and $\Pi^{(=k)}_m$ denote the natural projections of the IFSs $\mathcal{S}_{(<k),m}$ and $\mathcal{S}_{(=k),m}$.

  Let $\varrho^{(k)}_a, \eta^{(k)}_a$ be the self-similar measures of $\mathcal{S}_{(<k),m},\mathcal{S}_{(=k),m}$ respectively, both defined with uniform probabilities.
  With the help of these measures and \eqref{eq:plus}, we can write $\mu_{m,a}$ as a convolution
  \begin{equation}\label{od56}
    \mu_{m,a}=\varrho^{(k)}_a * \eta^{(k)}_a.
  \end{equation}

  By Theorem \ref{od10}, there exists an exceptional set of parameters $\mathcal{E}^{\prime}\subset(1/2,1)$ with $\dim_{\rm H}\mathcal{E}^{\prime}=0$ such that for every $a\in(1/2,1)\setminus\mathcal{E}^{\prime}$ the measure $\eta^{(k)}_a$ has polynomial Fourier decay. That is, for suitable $\alpha=\alpha(a,k)>0, C>0$ constants and sufficiently big $t$ values
  \begin{equation}\label{od55}
    |\widehat{\eta}^{(k)}_a(t)|\leq C|t|^{-\alpha},
  \end{equation}
  where $\widehat{\eta}^{(k)}_a(t)$ is the Fourier transform of $\eta^{(k)}_a(t)$.

  Now we show that $\mathcal{S}_{(<k),m}$ satisfies the strong exponential separation condition for all parameters outside a set of zero Hausdorff dimension. First, we define a function
  {
  \begin{equation}\label{od54}
    \gamma(x)=x+S_{\tilde{i}_1\dots \tilde{i}_m}(0)\frac{\lambda^{k-1}}{1-\lambda^k},
  \end{equation}
  }
  where {$\tilde{i}_n=1$ if $n\leq m-\floor{mp}$, and $\tilde{i}_n=2$ if $n> m-\floor{mp}$.}
  Let $\widetilde{\mathcal{S}}_{(<k),m}$ be the IFS obtained by conjugating the functions of $\mathcal{S}_{(<k),m}$ by $\gamma$.
  \begin{equation}\label{od53}
    \widetilde{\mathcal{S}}_{(<k),m}=\bigg\{\gamma\circ g_{\pmb{\jjv}}\circ\gamma^{-1} \:\bigg\vert\: g_{\pmb{\jjv}}\in\mathcal{S}_{(<k),m} \bigg\}
  \end{equation}
  It is easy to see that for any $g_{\pmb{\jjv}}\in\mathcal{S}_{(<k),m}$
  {
  \begin{align*}
    \gamma\circ g_{\pmb{\jjv}}\circ\gamma^{-1}(x) &= \lambda^kx +\sum_{l=1}^{k-1}\lambda^{l-1}S_{\jjv_l}(0) + S_{\tilde{i}_1\dots\tilde{i}_m}(0)\lambda^{k-1} \\
    &= S_{\jjv_1}\circ\cdots\circ S_{\jjv_{k-1}}\circ S_{\tilde{i}_1\dots\tilde{i}_m}(x)
  \end{align*}
  }
  Therefore, $\widetilde{\mathcal{S}}_{(<k),m}$ is a subsystem of $\mathcal{S}_m^k=\{S_{\iiv_1}\circ\cdots\circ S_{\iiv_k} \:\big\vert\: l\in\{1,\dots,k\}: \iiv_l\in\mathcal{M}_m \}$, and as such, inherits the strong exponential separation when it holds. Thus, according to Theorem \ref{od97} and the discussion at the beginning of Section~\ref{od26}, there exists a set $\mathcal{E}^{\prime\prime}\subset(1/2,1)$ with $\dim_{\rm H}\mathcal{E}^{\prime\prime}=0$ for which $\mathcal{S}_a$ has strong exponential separation for all $a\in\left(\frac{1}{2},1\right)$. So in particular, any subsystem is strongly exponentially separated, and so 
  \[
    \forall a\in\left(\frac{1}{2},1\right)\setminus\mathcal{E}^{\prime\prime}:
    \widetilde{\mathcal{S}}_{(<k),m} \:\mbox{ satisfies the SESC.}
  \]
  Since $\mathcal{S}_{(<k),m}$ is a conjugate of $\widetilde{\mathcal{S}}_{(<k),m}$, it also satisfies the strong exponential separation condition for $a\in(1/2,1)\setminus\mathcal{E}^{\prime\prime}$. It follows from Theorem \ref{od07} that
  \begin{equation}\label{od52}
    \dim_{\rm H}\varrho_a^{(k)} =\min\left\{1,\frac{\log |\mathcal{M}_m|^{k-1}}{-\log |\lambda|^k}\right\}.
  \end{equation}

  Using \eqref{od55} and \eqref{od52}, we may conclude the proof by applying Theorem \ref{od09} if we show that the right-hand side is greater than $1$ when $m$ and $k$ are sufficiently large. By the definition of $\mathcal{S}_{(<k),m}, \lambda$ and $p$, indeed
  {
  \[\begin{split}
     \lim_{\substack{m\to\infty\\k\to\infty}}\frac{\log |\mathcal{M}_m|^{k-1}}{-\log |\lambda|^k} &= \lim_{\substack{m\to\infty\\k\to\infty}}\frac{(k-1)\log\left(2^{m-\floor{pm}}\binom{m}{\floor{pm}}\right)}{-k\log \left(a^{m-\floor{pm}}(2a-1)^{\floor{pm}}\right)}\\
     &=
    \frac{-p\log p -(1-p)\log\frac{1-p}{2}}{-p\log (2a-1)-(1-p)\log a}>1.
  \end{split}\]
  }

 % \begin{gather*}
 %   \dim_{\rm Sim}\Lambda_{(<k),m} = \frac{\log |\mathcal{M}_m|^{k-1}}{-\log |\lambda|^k} = \frac{(k-1)\log |\mathcal{M}_m|}{-k\log |\lambda|},\\
 %   \lim_{k\to\infty}\dim_{\rm Sim}\Lambda_{(<k),m} = \frac{\log |\mathcal{M}_m|}{-\log |\lambda|} =
 %   \frac{\log\left(2^{\floor{(1-p)m}}\binom{m}{\floor{pm}}\right)}{-\log \left(a^{\floor{(1-p)m}}(2a-1)^{\floor{pm}}\right)}, \\
 %   \lim_{m\to\infty}\lim_{k\to\infty}\dim_{\rm Sim}\Lambda_{(<k),m} =
 %   \frac{-p\log p -(1-p)\log\frac{1-p}{2}}{-p\log (2a-1)-(1-p)\log a}>1.
 % \end{gather*}

\end{proof}

\begin{lemma}\label{od16}
  Let $\mathcal{E}\subset\left(\frac12,1\right)$ be the set as in Proposition~\ref{od57}, { and let $a\in\left(\frac12,1\right)\setminus \mathcal{E}$}. Then for every $\varepsilon>0$ there exists $M>0$ such that for every $m>M$
  \begin{equation}
    \dim_{\rm H}L_y \geq s_0-1-\varepsilon\text{ for $\mathcal{L}^1$-almost every $y\in\mathcal{K}_m$}.
    \end{equation}
\end{lemma}

\begin{proof}
Consider the subsystem of the Okamoto IFS $\mathcal{F}_a$ defined by $\mathcal{M}_m$
\[
  \mathcal{F}_m=\{f_{\iiv}\}_{\iiv\in\mathcal{M}_m}.
\]
Recall that $\nu_{m,a}$ denotes the uniform measure on $\mathcal{M}_m$.
Let $\Pi^{\mathcal{F}}_m:\Sigma\to[0,1]^2$ be the natural projection with respect to $\mathcal{F}_m$, and let $\widetilde{\mu}=\widetilde{\mu}_{m,a}$ be the push-forward of $\nu_{m,a}$ with respect to $\Pi^{\mathcal{F}}_m$.

Observe that $(\proj_{\sf y})_\ast\widetilde{\mu}=\mu_{m,a}$.
It follows from Proposition \ref{od57} that $\mu_{m,a}\ll\mathcal{L}^1$, and in particular, $\dim_{\rm H}\mu_{m,a}=1$.
By applying Theorem \ref{od15} to this measure, we obtain
\begin{equation}\label{od50}
  \dim_{\rm H}\widetilde{\mu}_y^{\proj_{\sf y}^{-1}} = \dim_{\rm H}\widetilde{\mu} -1, \mbox{ for $(\proj_{\sf y})_\ast\widetilde{\mu}$-almost every $y$},
\end{equation}
\begin{equation}\label{od49}
  \dim_{\rm H}\widetilde{\mu}= 1+\frac{h_{\widetilde{\mu}}-\chi_1(\widetilde{\mu})}{\log 3}.
\end{equation}
Further, by the construction of $\mathcal{M}_m$, it also follows that
\begin{equation}
  \lim_{m\to\infty}\dim_{\rm H}\widetilde{\mu} = s_0.
\end{equation}

Formulas \eqref{od50} and \eqref{od49} together imply
\begin{equation}\label{od48}
    \forall \varepsilon>0, \exists M>0, \forall m>M: \dim_{\rm H}L_y \geq s_0-1-\varepsilon,
\end{equation}
for $\mu_{m,a}$-almost every $y\in\mathcal{K}_m$.

According to \cite[Proposition~3.1.4]{barany2023self}, the measure $\mu_{m,a}$ is equivalent to the restriction of $\mathcal{L}^1$ to $\mathcal{K}_m$. Thus, \eqref{od48} also holds for $\mathcal{L}^1$-almost every $y\in\mathcal{K}_m$.

\end{proof}

\begin{proof}[Proof of Theorem {\ref{od60}}]
  We argue by contradiction, and assume that
  \begin{equation}\label{od47}
    \mathcal{L}^1(y\in[0,1]: \dim_{\rm H} L_y<s_0-1)>0.
  \end{equation}
  For every $\varepsilon>0$ we define the set
  \begin{equation}\label{od46}
    \mathrm{Bad}_{\varepsilon}:= \{y\in[0,1]: \dim_{\rm H} L_y<s_0-1-\varepsilon\}.
  \end{equation}
  Then, there must exist a small $\varepsilon>0$ for which
  \begin{equation}\label{od45}
    \mathcal{L}^1(\mathrm{Bad}_{\varepsilon})>0.
  \end{equation}
  By the Lebesgue density theorem, for every $\varepsilon^{\prime}>0$ and $\mathcal{L}^1$-almost every $y\in\mathrm{Bad}_{\varepsilon}$
  \begin{equation}\label{od44}
    \exists R>0, \forall r<R:
    \mathcal{L}^1(B(y,r)\cap\mathrm{Bad}_{\varepsilon})>
    (1-\varepsilon^{\prime})\mathcal{L}^1(B(y,r)),
  \end{equation}
  where $B(y,r)$ denotes the closed ball of radius $r$ around $y$.

  {Fix an arbitrary $y_0\in\mathrm{Bad}_{\varepsilon}$.
  For any $r>0$, we can find an $\iiv\in\Sigma_\ast$ such that
  \begin{equation}\label{od43}
    S_{\iiv}[0,1]\subset B(y_0,r), \mbox{ and } \left|S_{\iiv}[0,1]\right|\geq (2a-1)r.
  \end{equation}
  Pick $\varepsilon^{\prime}<\frac{\mathcal{L}^1(\mathcal{K}_m)}{2}(2a-1)$.
  Choose an $m>1$ for which \eqref{od48} holds with $\frac{\varepsilon}{2}$, and choose an $r_0>0$ for which \eqref{od44} holds with $\varepsilon^{\prime}$.
  Further, let $\jjv\in\Sigma_\ast$ be the finite word that satisfies \eqref{od43} for $r_0$.

  Clearly, Lemma \ref{od16} implies that
  \[
    \forall \iiv\in\Sigma_\ast: \dim_{\rm H}L_{S_{\iiv}(y)} \geq s_0-1-\varepsilon,
  \]
  for $\mathcal{L}^1$-almost every $y\in\mathcal{K}_m$. Thus by definition,
  {
  \begin{equation}\label{od42}
    B(y_0,r_0)\cap \mathrm{Bad}_{\varepsilon}\cap S_{\jjv}(\mathcal{K}_m) =\emptyset.
  \end{equation}
  }
  However,
    \begin{align*}
      \mathcal{L}^1(S_{\jjv}(\mathcal{K}_m))\geq (2a-1)r_0\mathcal{L}^1(\mathcal{K}_m)
      =\left(\frac{(2a-1)\mathcal{L}^1(\mathcal{K}_m)}{2}\right)\mathcal{L}^1(B(y_0,r_0)), \\
      \mathcal{L}^1(B(y_0,r_0)\cap\mathrm{Bad}_{\varepsilon})>\left(1-\frac{(2a-1)\mathcal{L}^1(\mathcal{K}_{m})}{2}\right)\mathcal{L}^1(B(y_0,r_0)).
    \end{align*}
  As $S_{\jjv}(\mathcal{K}_m)\subset B(y_0,r_0)$ and
  \[
    \mathcal{L}^1(S_{\jjv}(\mathcal{K}_m)) + \mathcal{L}^1(B(y_0,r_0)\cap\mathrm{Bad}_{\varepsilon}) > \mathcal{L}^1(B(y_0,r_0)),
  \]
  they must intersect each other.
  }
  It contradicts \eqref{od42}, hence
  \[
    \mathcal{L}^1(y\in[0,1]: \dim_{\rm H} L_y<s_0-1)=0.  \]
\end{proof}

\bibliographystyle{abbrv}
\bibliography{Okamoto_bib}

\end{document}